\documentclass[twoside,reqno,10pt]{amsart}

\usepackage{amsmath, amssymb, amsfonts}
\usepackage{url}

\newcommand{\fdiff}[4]{\ensuremath{ \upsilon^{#4 }_{ #3 }  #1  \left(  #2 \right)  }}
\newcommand{\fdiffplus}[3]{ \fdiff {#1}{#2}{+}{#3} }

\newcommand{\fracvar}[4]{\ensuremath{ \upsilon_{ #3 }^{ #4} \left[ #1 \right] \left(  #2 \right)   }}

\newcommand{\llim}[3]{\ensuremath{ \lim\limits_{ #1 \rightarrow #2} #3 }}
\newcommand{\fclass}[2]{\ensuremath{  \mathbb{#1}^{\, #2} }}
\newcommand{\fcclass}[2]{\ensuremath{  \mathcal{#1}^{\, #2} }}

\newcommand{\gvarplus}[2]{ \fracvar {#1}{#2}{g}{\epsilon+ }}

\newcommand{\gvarpm}[2]{ \fracvar {#1}{#2}{g }{\epsilon\pm}}

\newcommand{\gendiff}[4]{\ensuremath{\mathcal{D}_{#1}^{#2} {#3} \left( #4 \right)  }}
\newcommand{\gdiffplus}[2]{ \gendiff {g}{+}{#1}{#2} }
\newcommand{\gdiffmin}[2]{ \gendiff {g}{-}{#1}{#2} }
\newcommand{\gdiffpm}[2]{ \gendiff {g}{\pm}{#1}{#2} }

\newcommand{\wdiffpm}[2]{ \gendiff {\omega}{\pm}{#1}{#2} }

\newcommand{\holder}[1]{\fclass{H}{#1} }

\newcommand{\osc}[4]{\ensuremath{ \mathrm{osc}_{ #3 }^{ #4} [ #1 ] \left(  #2 \right)   }}
\newcommand{\oscplus}[2]{ \osc {#1}{#2}{\epsilon}{+} }
\newcommand{\oscmin}[2]{ \osc {#1}{#2}{\epsilon}{-} }
\newcommand{\oscpm}[2]{ \osc {#1}{#2}{\epsilon}{\pm} }
\newcommand{\osci}[1]{ \ensuremath{\underset{#1}{\mathrm{osc}}\,  } }
\newcommand\smallO{
	\mathchoice
	{{\scriptstyle\mathcal{O}}}% \displaystyle
	{{\scriptstyle\mathcal{O}}}% \textstyle
	{{\scriptscriptstyle\mathcal{O}}}% \scriptstyle
	{\scalebox{.7}{$\scriptscriptstyle\mathcal{O}$}}%\scriptscriptstyle
}

\newcommand{\bigoh}[1]{ \ensuremath{ \smallO \left(  #1 \right) }   }
\newcommand{\bigohx}[1]{ \ensuremath{ \smallO_{  #1} }   }

\newcommand{\deltaop}[4]{\ensuremath{ \Delta_{ #3 }^{ #4} \left[ #1 \right] \left(  #2 \right)   }}
\newcommand{\deltaplus}[2]{ \deltaop {#1}{#2}{\epsilon}{+} }
\newcommand{\deltamin}[2]{ \deltaop {#1}{#2}{\epsilon}{-} }
\newcommand{\deltapm}[2]{ \deltaop {#1}{#2}{\epsilon}{\pm} }
\newcommand{\rref}[1]{ \ensuremath{\mathrm{\ref{#1}} }}

\newcommand{\modcont}[2]{\ensuremath{\omega_{#1} \left( #2 \right)   }}

\newcommand{\predf}[2]{\ensuremath{\mathcal{ #1} \left[ #2 \right]  }}
\newcommand{\partd}{\ensuremath{\mathcal{ P} }}

\newcommand{\tth}{\textsuperscript{th}\;}

\newtheorem{theorem}{Theorem}
\newtheorem{lemma}{Lemma}
\newtheorem{corollary}{Corollary}
\newtheorem{proposition}{Proposition}

\newtheorem{condition}{Condition}

\newtheorem{definition}{Definition}
\newtheorem{remark}{Remark}
\newtheorem{example}{Example}
\newcommand{\ubar}[1]{\underaccent{\bar}{#1}}

\newcommand{\bigohone}[1]{ \ensuremath{ \smallO_{  #1 }  }   }

\newcommand{\funct}[3]{ \ensuremath{ #1 : \mathbb{ #2} \mapsto \mathbb{ #3}    }}

\usepackage{graphicx,color}
\usepackage{accents}
\begin{document}

% \begin{frontmatter}
 	
 	\title{Local generalizations of the derivatives on the real line}
 	
%\author[aaa,bbb]{Dimiter Prodanov }
 	\author{Dimiter Prodanov}
	\address{EHS, Kapeldreef 75, IMEC, 3001, Leuven, Belgium}

%\ead{Dimiter.Prodanov@imec.be, dimiterpp@gmail.com}
%\textsuperscript{a,b,*}
%\address[aaa]{Environment, Health and Safety, Leuven, IMEC, Leuven, Belgium}
%\address[bbb]{MMSIP, IICT, BAS, Bulgaria}

%\cortext[cor1]{correspondence: IMEC vzw, Kapeldreef 75, 3001 Leuven, Belgium}

%\author[mymainaddress, mysecondaryaddress]{Dimiter Prodanov}
%\ead{Dimiter.Prodanov@imec.be, dimiterpp@gmail.com}

%\address[mymainaddress]{Environment, Health and Safety, Leuven, IMEC, Leuven, Belgium\corref{cor1}}
%\address[mysecondaryaddress]{MMSIP, IICT, BAS, Bulgaria}
  
%\cortext[cor1]{correspondence: IMEC vzw, Kapeldreef 75, 3001 Leuven, Belgium}

	%%%%%%%%%%%%%%%%%%%%%%5
\begin{abstract}
		From physical perspective, derivatives can be viewed as mathematical idealizations of the linear growth. 
		The linear growth condition has special properties, which make it preferred.
		The manuscript investigates the general properties of the local generalizations of derivatives assuming the usual topology of the real line.  
		The concept of derivative is generalized in terms of the class of the modulus of continuity of the primitive function.
		This definition is suitable for applications involving continuous but possibly non-absolutely continuous functions of a real variable. 
		The main application of the approach is the generalization of the Lebesgue monotone differentiation theorem.
		On the second place, the conditions of continuity of generalized derivative are also demonstrated.	
\end{abstract}
	
% Keywords
% \begin{keyword}
%non-differentiable functions \sep
%singular functions \sep  
%moduli of continuity \sep
%H\"older classes 
%\MSC[2010] Primary 26A27: Secondary 26A30, 26A15, 26A33,  26A16, 7A52, 4102
%\end{keyword}

%\end{frontmatter}
\maketitle


  1) Environment, Health and Safety, IMEC, Belgium 2) MMSDP, IICT, BAS, Bulgaria

Correspondence Address: IMEC vzw, Kapeldreef 75, 3001 Leuven, Belgium 
	%%%%%%%%%%%%%%%%%%%
	% Section
	%%%%%%%%%%%%%%%%%%%
	\section{Introduction}
	\label{sec:intro} 
	
	Since the time of Newton, it is accepted that celestial mechanics and physical phenomena are, by and large, described by smooth and continuous functions.
	The second law of Newton demands that the velocity is a differentiable function of time. 
	This ensures mathematical modelling in terms of differential equations, and hence (almost everywhere) differentiable functions. 
	Ampere even tried to prove that all functions are almost everywhere differentiable. 
	Now we know that this attempt was doomed to fail.

	Various non-differentiable functions have been constructed in the XIX\tth  century and regarded with a mixture of wonder and horror.	
	The interest in fractal and non-differentiable functions was rekindled with the works of Mandelbrot in fractals \cite{Mandelbrot1982}. 
	For example, Fonf et al. (1999)  have established that there is a closed, infinite-dimensional subspace of $ C[0, 1]$ consisting of nowhere differentiable functions\cite{Fonf1999}. 
	This existence result was extended with constructive proofs in \cite{Girgensohn2001, Berezhnoi2003}. 
	
	Scientific developments in the last 50 years indicate that
	the use of non-differentiable functions can not be avoided when modelling nature. 
	For instance, it is easy to establish that stochastic paths of the classical Wiener process are non-differentiable.
	Applications of this process are ubiquitous in physics, biology and economy. 
	On a second place, some physical theories also consider non-differentiable functions. 
	The stochastic interpretation of Quantum Mechanics, introduced by Nelson  \cite{Nelson1966}, assumes a reversible sub-quantum Brownian motion (i.e. reversible Wiener process), having non-differentiable trajectories. 
	In a closely related manner, almost all, in the measure sense, paths in the formulation of the Feynman path-integral are non-differentiable \cite{Feynman1948}.  
	The deterministic approach of scale relativity theory, introduced by Nottale \cite{Nottale2010} also assumes non-differentiability of the fundamental space-time manifold. 
	The Ornstein--Uhlenbeck (OU) process was introduced in the kinetic theory of gasses \cite{Uhlenbeck1930}. 
	In this process, the particle velocities are non-differentiable.  
	The OU process arises as the scaling limit of the Ehrenfest urn model which describes the diffusion of particles through a permeable membrane.

 	Purely mathematically, the derivatives can be generalized in several ways. 
 	Derivatives can be defined in the usual way as limits of difference quotients on the accumulation sets of points \cite[ch 3, p. 105]{Milanov1977}.
 	This approach can be applied also to functions defined on fractal sets \cite{Yang2015}. On the other hand, the question of continuity on intervals of so-defined functions requires further specification. 
 	
 	If continuity is perceived as an essential property such generalization leads to various integro-differential operators.
 	The best known examples here are the Riemann-Liouville and Caputo operators.
 	However, such operators lead to non-local (interval) functions.
 	Application of a subsequent limiting localization operation can lead to a local operator.
 	An example of this is the local fractional derivative introduced by Kolwankar and Gangal \cite{Kolwankar1997}:
 	\[
  \mathcal{D}_{KG}\, f (x) :=	\llim{x}{a}{}\frac{1}{\Gamma(1- \beta)}  \frac{d}{dx}  \int_{a}^{x}\frac{  f (t ) }{{\left( x-t\right) }^{\beta }}dt 
 	\]
 	However, such localization can only lead to trivially continuous functions -- that is -- the result of the localization is zero where the derivative function is continuous \cite{Prodanov2018}. 
	
	Point-wise, the derivatives can be generalized by formal "fractionalization" -- i.e. by the substitution $\epsilon \to \epsilon^\beta$  leading to 
	\[
  \fdiffplus{ f}{x}{\beta}:=\llim{\epsilon}{0}{\frac{f(x+ \epsilon) - f(x) }{\epsilon ^\beta}} 
	\] 
	The quantity in this expression is called \textit{fractional velocity}.
	Such an approach has been considered for the first time by the mathematicians Paul du Bois-Reymond and Georg Faber in their studies of the point-wise differentiability of functions \cite{BoisReymond1875, Faber1909}. 
	In the late  XX\tth century, the physicist Guy Cherbit introduced the same quantity  under the name \textsl{fractional velocity} by analogy with the Hausdorff dimension as a tool to study the fractal phenomena and physical processes for which  instantaneous velocity was not well-defined \cite{Cherbit1991}. 
	The properties of fractional velocity have been extensively studied in \cite{Adda2001} and \cite{Prodanov2017}. 
	The special choice of the function $\epsilon^\beta$ can be justified from the theory of the  fractional calculus as the limit of the regularized Riemann-Liouville differ-integral (see above). 
	  
	As can be expected, the overlap of the definitions of the Cherebit's fractional velocity and the Kolwankar-Gangal local fractional derivative is not complete. 
	The precise equivalence conditions have been established elsewhere  \cite{Adda2013, Prodanov2018}. 
	Both definitions are closely linked with conditions for the growth of the functions.
	Notably, Kolwankar-Gangal fractional derivatives are sensitive to the critical local H\"older exponents, while the fractional velocities are sensitive to the critical point-wise H\"older exponents and there is no complete equivalence between those quantities \cite{Kolwankar2001}.  
	
	In contrast to physical applications, mathematically, there is no reason to limit the choice of the function in the denominator of the difference quotient to a power function. 
	In such way, more diverse limit objects generalizing derivatives can be studied.
	Such is the objective of the present paper. 
	Here, derivatives are generalized in terms of the class of the \textit{modulus of continuity} of the primitive function.
	Such definition focuses on applications involving continuous but possibly non-absolutely continuous functions of a real variable. 
	
	The paper is structured  as follows:
	Section \ref{sec:nullset} discusses some properties of totally disconnected sets.
	Section \ref{sec:osc} introduces point-wise oscillation of functions. 	
	Section \ref{sec:contmod} characterizes some of the properties of the moduli of continuity.
	Section \ref{sec:genderivomega} introduces the concept of generalized $\omega$-derivatives, defined from the maximal modulus of continuity. 
	Section \ref{sec:contderiv} discusses the continuity sets of derivatives form the perspective of the theory developed in Sec. \ref{sec:osc} and \ref{sec:contmod}. 
	Section \ref{sec:genderiv} introduces the concept of modular derivatives. 

	Notational conventions (see Appendix \ref{sec:definitions}) follow previously published works and are repeated here for convenience \cite{Prodanov2017,Prodanov2018, Prodanov2019}. 
	The notation $|I|$ for an interval $I$ will mean its length.
	 
	%%%%%%%%%%%%%%%%
	% Section
	%%%%%%%%%%%%%%%%%
	\section{Totally disconnected sets}\label{sec:nullset}	
	 The following definition is given in Bartle (2001)\cite[Part 1, Ch. 2]{Bartle2001} and Silva (2007) \cite[Ch. 2]{Silva2007}:
	 %%%%%%%%%%%%%%%
	 %  Def
	 %%%%%%%%%%%%%%%
	 \begin{definition}[Null sets]\label{def:nullsetf}
	  A \textbf{null set} $Z \subset \fclass{R}{}$ (or a set of measure 0) is called a set, such that for every $0<\epsilon <1$ there is a countable collection of sub-intervals $\left\lbrace I_k \right\rbrace^\infty_{k=1}  $,  such that
	 \[
	 	Z \subseteq \bigcup\limits_{k=1}^{\infty} I_k , \quad  \sum\limits_{k=1}^{\infty} |I_k| \leq \epsilon 
	 \]
	 where $|.|$ is the interval length. 
	 Then we write $|Z|=0$.
	 \end{definition} 
	 %%%%%%%%%%%%%%%%
 
 	%%%%%%%%%%%%%%%
 	% remark
 	%%%%%%%%%%%%%%
 	\begin{remark}
 		The next statement is a generalization of a well-known property of  countable sets (see for example \cite[Ch. 2]{Silva2007}). The result is given here for completeness of the subsequent presentation. 
 	 	\end{remark}
	 %%%%%%%%%%%%%%

	 %%%%%%%%%%%%%%
	 %  Def
	 %%%%%%%%%%%%%%
	 \begin{definition}[Totally disconnected space]
	A metric space \textit{M} is totally disconnected if every non-empty connected subset
	of \textit{M} is a singleton \cite[p. 210]{Willard2004}.
	That is, for every $S \subset M,$ \textit{S} non-empty and connected
	implies $ \exists p \in M$ with $S = \{p\}$.
	 \end{definition}
 	%%%%%%%%%%%%%%

	 %%%%%%%%%%%%
	 %  theorem
	 %%%%%%%%%%%%
	 \begin{theorem}[Null set disconnectedness]\label{th:disconect}
	 	Suppose that $E$ is a null set.  Then $E$ is totally disconnected.
	 	Conversely, suppose that $E$ is totally disconnected and countable.
	 	Then $E$ is a null set.
	 \end{theorem}
	 %%%%%%%%%%%%%%%%%%%%%%
	 \begin{proof}
	 	\begin{description}
	 		
 		%%%%%%%%%%%
 		% FWD
 		%%%%%%%%%%
 		\item[Forward statement]  
	 	Suppose that $Z \subset E$ is connected and open. 
	 	Then there exist 3 numbers $x_1 < z < x_2$, such that $[x_1, x_2] \subset Z$.
	 	Then $|[x_1, x_2] | = x_2 - x_1 >0$.
	 	Therefore, $\exists\epsilon$, such that $ 0< \epsilon \leq z - x_1 < x_2 - x_1$; so that
	 	$\epsilon < |Z| \leq | E|$, which is a contradiction.
	 	Therefore, $x_2=x_1$ and hence $Z$ is singleton. 
	 	Therefore, by induction $E$ is totally disconnected.  
	 	
	 	%%%%%%%%%%%%
	 	% CONV
	 	%%%%%%%%%%%% 	 
	 	\item[Converse statement]
	 	
	 	The countability requirement in the statement of the theorem comes from the fact that there are sets that are totally disconnected, uncountable and non-null  \cite{Kantor2015, Flood2011} (see   example \ref{ex:SVCset}).
	 	Since $E$ is totally disconnected for every $z, w \in E$, trivially, there is a number $h$, such that  
	 	$[z-h/2, z+ h/2] \cap [w-h/2, w+ h/2] =\emptyset  $.
	 	Therefore, there is a  collection of such intervals,
	 	$ \{I_k\}_{k=1}^{\infty}$ 
	 	%$
	 	%I_k=[z_k-h/2, z_k+ h/2] /2^k 
	 	%$, that is
	 	 $$I_k=[z_k -h/2^{k+1},   z_k +h/2^{k+1} ]$$
	 	 of length $|I_k|=1/2^k$.
	 	Therefore,
	 	$$\sum\limits_{k=1}^{\infty} |I_k| = h$$
	 	 for any such a number $h$.
	 	 Since $h$ can be chosen arbitrarily small the claim follows.
	 	 
	 	 	\end{description}
	 \end{proof}
	 %%%%%%%%%%%%%%%%%%%%%%%%
	It should be noted that not all totally disconnected sets are null.
	There are totally disconnected uncountable sets of positive measure.   
	For example, the  Smith--Volterra--Cantor set is of Lebesgue measure $1/2$. 
 	%%%%%%%%%%%%%%%%%%%%%%%
 	\begin{example}\label{ex:SVCset}
 		The construction of the Smith--Volterra--Cantor set is given as follows \cite[p.15]{Kantor2015}:
 		The set is constructed by iteratively removing certain intervals from the unit interval $I_0=[0, 1]$.
 		At each step $k$, the length that is removed $p_{k+1}= p_k/4$ from the middle of each of the remaining intervals.
 		That is, starting from $I_0$  and $p_0=1/4$
 	    on every step
 		\begin{flalign*}
 		I_k=[u, v]  & \longrightarrow I_{k+1}^l= \left[ u, (u+v)/2-p_{k}/2  \right], \quad
 		I_{k+1}^r= \left[  (u+v)/2+p_{k}/2, v  \right] \\
 		p_k & \longrightarrow p_{k+1}= p_k /4
		\end{flalign*}
		For example, 
		 \[
		 \begin{array}{ll}
		 k = 1 : &
		 I_1=\left[ 0,\frac{3}{8}\right], \ I_2 = \left[ \frac{5}{8},1\right] \\
		 ~\\
		 k = 2 : &
		 I_{21}=\left[0,\frac{5}{32}\right], \ I_{22}=\left[\frac{7}{32},\frac{3}{8}\right],
		 I_{23}=\left[\frac{5}{8},\frac{25}{32}\right], \ 
		 I_{24}=\left[\frac{27}{32},1\right] \\
		 ~\\
		 		k = 3 : & 
		 I_{31}=\left[0,\frac{9}{128} \right], \ I_{32}=\left[\frac{11}{128},\frac{5}{32}\right],
		 \ I_{33}=\left[\frac{7}{32},\frac{37}{128}\right], \ I_{34}=\left[\frac{39}{128},\frac{3}{8}\right], \\
		 ~\\
		 &\ I_{35} =\left[\frac{5}{8},\frac{89}{128}\right],\ I_{36}=\left[\frac{91}{128},\frac{25}{32}\right], \ 
		 I_{37}=\left[\frac{27}{32},\frac{117}{128}\right], \ 
		 I_{38}=\left[\frac{119}{128},1\right]	
		 \end{array}
		 \]
	  During the process, disjoint intervals of total length
	  \[
	  L=\sum\limits_{k=0}^{\infty} \frac{1}{4 \, .\, 2^k } = \frac{1}{2}
	  \]
	  are removed so that the resulting set is of measure $1/2$.
	  The Smith--Volterra--Cantor set is closed as it is an intersection of closed sets.
	  Furthermore, at step $n$ the length of each closed subinterval is $l_n=\frac{1}{2} \left(l_{n-1}- p_{n-1} \right) $.
	  Starting from $l_0=1$ one gets  
	  $$
	  l_n=\frac{1}{2}\left(\frac{1}{2^{n}}+\frac{1}{4^{n}} \right) 
	  $$
	  Therefore, by the Nested Interval theorem the SVC set is totally disconnected and contains no intervals.
	  The SVC set was used as an example in \cite{Prodanov2019}.
 	\end{example}
 	%%%%%%%%%%%%%%%%%%%%%
 	The set presented in the above example can be used to construct a singular function, resembling by some of its properties the famous "Devil's staircase" function (see  \cite{Prodanov2019}). 
 
	 %%%%%%%%%%%%%%%%
	 % Section
	 %%%%%%%%%%%%%%%%%
	 \section{Interval and Point-wise Oscillation of Functions}
	 \label{sec:osc}
	 
	 The concept of point-wise oscillation can be used to characterize the set of continuity of a function.
	 This can be done in a way similar to the approach presented as theorem 3.5.2 in Trench \cite{Trench2013}[p. 173]. 
	 This is the so-called Oscillation lemma \cite{Prodanov2017,Prodanov2019}. Since it was published before, the statement of the lemma is relegated to an Appendix. 
	 %%%%%%%%%%%%%%%%%%%%%%
	 %  Def
	 %%%%%%%%%%%%%%%%%%%%%%
	 \begin{definition}[Oscillation]
	 	\label{def:limosc1}
	 	Define the oscillation of the function \textit{f} in the interval $J=[a, b]$ as
	 	\[
	 	\osci{J} f:= \sup_{J} f - \inf_{J} f
	 	\]
	\end{definition}
	 %%%%%%%%%%%%%%%%%%%%%%
	 %  Def
	 %%%%%%%%%%%%%%%%%%%%%%
	 \begin{definition}[Directed Oscillation]
	 	\label{def:limosc}
	Define the directed oscillations as 
   	(i) the forward oscillation:
	\begin{flalign*}
		\oscplus{f}{x}  : =  & \sup_{[x , x + \epsilon]} { f} -	\inf_{[x , x + \epsilon]} { f}, \quad  I=[x , x + \epsilon]
	\end{flalign*}
	and the backward oscillation:
	\begin{flalign*}
			\oscmin{f}{x}   : =  & \sup_{[x - \epsilon , x ]} {f} - \inf_{[x -   \epsilon, x ]} { f}, \quad  I=[ x - \epsilon, x]
	\end{flalign*}
	Finally, define the limits, if such exist as finite numbers, as
	\begin{flalign*}
	\mathrm{osc^{+}} [f] (x) : = &  \llim{|I|}{0}{} 
	\left( \sup_{I} f - \inf_{I} f \right), \quad  I=[x , x + \epsilon] \\
			\mathrm{osc^{-}} [f] (x) : = &    \llim{|I|}{0}{} 
	\left( \sup_{I} f - \inf_{I} f\right), \quad  I=[ x - \epsilon, x]
	\end{flalign*}
	according to previously introduced notation \cite{Prodanov2017,Prodanov2018}.
	 \end{definition}
	 %%%%%%%%%%%%%%%%%%%%%%%%%%%%

	The Oscillation lemma is of a fundamental importance for it opens up the possibility to characterize the discontinuity of functions in terms of their oscillation at a given point.
	The oscillation of a function can be viewed in two ways: 
	as a functional having the interval of study fixed; 
	or, alternatively, as a function of the interval having the function under study fixed.
	There is no ambiguity as in fact both aspects are complementary as will be demonstrated.
	
	%%%%%%%%%%%%%
	% Def
	%%%%%%%%%%%%%
	\begin{definition}[super/sub-additivity on an interval]\label{def:supadd}
		A function $f$ is called sub-additive on the interval $I=[x, x+ \epsilon]$ if
		\[
		f(x+ a) +f(x+b) \geq f(x+ a+b), \quad a,b \in [0, \epsilon] 
		\]
		The converse holds for super-additivity 
		\[
		f(x+ a) +f(x+b) \leq f(x+ a+b) 
		\]
	\end{definition}
	%%%%%%%%%%%%%%
	\begin{example}
		$f(x)=x^3$ is sub-additive in $ (-\infty, 0)$ and super-additive in  $(0, \infty)$.
		Let $a,b>0$. Then
		\[
		a^3 +b^3 \leq (a+b)^3
		\]
		since
		\[
		0 \leq 3\, ab (a+b)
		\]
		and hence super-additivity follows on the positive real axis.  
	\end{example}
	The above definition allows one to establish some properties of the oscillation.
	The subsequent lemma is useful for that purpose:
	
	%%%%%%%%%%%%%%%
	%  Lemma
	%%%%%%%%%%%%%%%
	\begin{lemma}\label{th:superaddtive}
		Let $f$ be a non-decreasing function on $I=[x, x+ \epsilon]$.
		If $f$ is also super-additive on $I$ then 
		\[
		\sup_{ A} f + \sup_{ B} f \leq \sup_{ I} f
		\]
		for $   A=[x,x +a] \subset I $, $B=[x,x +b] \subset I$, $A \cup B= I$.
		
		Conversely, if $f$ is increasing and sub-additive on $I$ then 
		\[
		\sup_{ A} f + \sup_{ B} f \geq \sup_{ I} f
		\]
	\end{lemma}
	%%%%%%%%%%%%%%%
	\begin{proof}
		Let $I=[x, x+a+b]$, $A=[x,x +a]$, $B=[x,x +b]$.
		Let $f$ be super-additive and non-decreasing on $I$.
		Consider the following table of values
		\[
		\begin{matrix}{}
		a^\prime \leq & a \leq & b^\prime < & b  < & c^\prime < & a+b \\
		\sup_{ A} f = f(a^\prime) \leq & f(a) \leq & \sup_{ B} f =f (b^\prime) \leq  & f (b) \leq & \sup_{ I} f = f(c^\prime) \leq & f(a+b)
		\end{matrix}
		\]
		By the non-decreasing property
		\[
		M:= f(a^\prime) + f (b^\prime) - f(c^\prime) \leq f(c^\prime)
		\]
		Then also
		\[
		M \leq f(a) + f (b) - f(c^\prime) \leq  f(a + b) - f(c^\prime) \leq 0
		\]
		by super-additivity. 
		Therefore, 
		\[
		\sup_{ A} f + \sup_{ B} f \leq \sup_{ I} f
		\]
		
		Let $f$ be sub-additive and increasing on $I$.
		It follows also that
		\[
		M \geq f (a^\prime+b^\prime) - f (c^\prime) 
		\]
		However, since the function is bounded $a=a^\prime$, $b=b^\prime$, $c^\prime= a+b$. Therefore, $M \geq 0$ and
		\[
		\sup_{ A} f + \sup_{ B} f \geq \sup_{ I} f
		\]	
	\end{proof}
	%%%%%%%%%%%%%%%
	
	%%%%%%%%%%
	% Example
	%%%%%%%%%%
	\begin{example}\label{ex:1}
		Let $f(x)= x^2$ and $I=[0, a+b]$, $0<a<b$.
		Then $\sup_{ A} f=a^2$, $\sup_{ B}=b^2$, $\sup_{ I} f = (a+b)^2$ and
		\[
		a^2+b^2 \leq (a+b)^2
		\]
		which is true.
		
		Let $f(x)= \sqrt{x}$. Then $\sup_{ A} f=\sqrt{a}$, $\sup_{ B}=\sqrt{b}$, $\sup_{ I} f = \sqrt{a+b}$ and
		\[
		\sqrt{a} + \sqrt{b} \geq \sqrt{a+b}
		\]
		which is true.
	\end{example}
	%%%%%%%%%%%%%%
	\begin{lemma}\label{th:convsuperadd}
		Consider a function, which is sub-additive on $I=[x,x +\epsilon]$. 
		Then $f$ is concave on $I$. That is, for $0\leq \lambda \leq 1$
		\[
		f (\lambda a+ \left(1-  \lambda \right)  b) \geq \lambda \; f (a) + \left(  1 - \lambda \right)   f (b)
		\]
		holds for any $ a+b \leq \epsilon,a,b>0$.
		Conversely, if $f$ is super-additive on $I$ then it is convex:
		\[
		f (\lambda a+ \left(1-  \lambda \right)  b) \leq \lambda \; f (a) + \left(  1 - \lambda \right)   f (b)
		\]
	\end{lemma}
	%%%%%%%%%%%%%%%%%%%
	\begin{proof}
		We prove first the sub-additive case.
		Consider the integer $k \geq 1$. Then for some real $a$ it follows by induction that
		$f( k a ) \leq k \; f(a)  $. 
		Further, suppose that $a= b/k$ for some $b$. Then
		$
		f(  b )/k \leq  f(b/k) 
		$
		so that combining with the previous inequality it follows that for a rational number $q=p/k \geq 1$ :
		$f( q a ) \leq q \; f(a)  $.	
		Let $r=1/q$ and $b =a/r$; then 
		$  f( q a ) \leq q \; f(a) \Longrightarrow  r \; f( r ) \leq   f( b \, r)$ for $ r  \leq 1  $.
		Since $a$ is arbitrary  then  the inequality is valid for any $b>0$.
		
		Letting $a= \lambda/q$, $\lambda \in \fclass{R}{}$ it follows that
		$f(  \lambda )  \leq \lambda/a \; f( a)  $ for $\lambda \geq 1$.
		Since now both variables are real-valued the entire domain becomes real. 
		
		The concavity of $f$ is established as follows:
		Let $ \lambda = a/ (a+b) \leq 1$ and the opposite be assumed true.
		\begin{multline*}
		f \left( a \right) +   f \left( b \right) =   f \left(    \frac{a }{a+b}  \left( a+b \right) \right)  + f  \left( \frac{b  }{a+b} \left( a+b \right) \right) =  \\
		f  \left( \lambda (a+b)\right)  + f  \left( (1-\lambda) (a+b) \right) 
		\end{multline*}
		So that
		\[
		f  \left( \lambda (a+b) \right)  + f  \left( (1-\lambda) (a+b) \right)  \leq \lambda \; f \left(   a+   b \right) + \left(1- \lambda \right) f\left(   a+   b \right)= f\left(   a+   b \right)
		\]
		Therefore, 
		$f \left( a \right) +   f \left( b \right) \leq f\left( a+   b \right)  $, which is a contradiction to the initial hypothesis. Therefore, $f$ is concave on $I$. 
		
		The super-additive case can be proven in the same way and holds by duality. 
	\end{proof}
	%%%%%%%%%%%%%
	The next point is to establish the properties of oscillation.
	The reasoning is symmetric with regard to sub-additive and super-additive functions.	
	%%%%%%%%%%%%%%%%%%%5
	%%%%%%%%%%%%
	% Theorem
	%%%%%%%%%%%%
	\begin{lemma}\label{th:oscsupadd}
		Suppose that \textit{f} is non-negative and super-additive on $I=[x,x+\epsilon]$ and  $\inf_{ I}=0$.
		Then
		\[
		\osci{a}f (x) + \osci{b} f (x) \leq\osci{a+b \leq \epsilon} f
		\]
	\end{lemma}
	%%%%%%%%%%%%%%%%%
	\begin{proof}
		Let $A \subset B \subset I$.
		Suppose that \textit{f} is non-negative and super-additive.
		\[
		\osci{A} f + \osci{B} f -\osci{I} f = \sup_{ A} f + \sup_{ B} f -\sup_{ I} f  -\inf_{ A} f - \inf_{ B} f + \inf_{I} f \leq
		\inf_{I} f -\inf_{ A} f - \inf_{ B} f \leq 0
		\]
		Therefore,
		\[
		\osci{A} f + \osci{B} f \leq \osci{I} f
		\]
	\end{proof}
	%%%%%%%%%%%%%%%%%%%
	%%%%%%%%%%%%
	% Theorem
	%%%%%%%%%%%%
	\begin{lemma}\label{th:oscsuperadd}
		Suppose that \textit{f} is increasing and sub-additive on $I=[x,x+\epsilon]$ and  $\inf_{ I}=0$.
		Then
		\[
		\osci{a}f (x) + \osci{b} f (x) \geq \osci{a+b \leq \epsilon} f
		\]
	\end{lemma}
	\begin{proof}
		Let $A \subset B \subset I$.
		Suppose that \textit{f} is increasing and sub-additive.
		\begin{multline*}
		M:=\osci{A}f + \osci{B} f -\osci{I} f = \sup_{ A} f + \sup_{ B} f -\sup_{ I} f  -\inf_{ A} f - \inf_{ B} f + \inf_{I} f \geq \\
		\inf_{I} f -\inf_{ A} f - \inf_{ B} f =: N
		\end{multline*}
		However, $N \leq 0$, since $\inf_{I} f \leq \inf_{ A} f + \inf_{ B} f$ by hypothesis.
		Therefore, if $\inf_{ I}=0$
		\[
		\osci{A}f + \osci{B} f -\osci{I} f \geq 0
		\]
	\end{proof}
	%%%%%%%%%%%%%%%%
	From this last result it is clear that the condition $\inf_{ I} f=0$ is not  attainable in general.
	Moreover, the use of the \textit{infimum} brings also another point in the interval into consideration,
	for trivially
	\[
		\osc{f}{x}{\epsilon}{ +} = \sup_{u ,v \, \in I, u \neq v} \left| f(u) - f(v) \right| ;
	\]
	therefore, the \textit{oscillation}, defined as above, includes information about two points \textit{u} and \textit{v} in some relation to the point of interest \textit{x}.
	Therefore, another function will be introduced that maintains the desirable properties established in Lemma \ref{th:superaddtive}.

	%%%%%%%%%%%%%%
	%  Section
	%%%%%%%%%%%%%%%
	\subsection{The Point Oscillation Function}\label{sec:oscfunct}
	
	%%%%%%%%%%%%%
	% Def
	%%%%%%%%%%%%%
	\begin{definition}\label{def:oscfunct}
		Consider a bounded function $f$ defined on a compact interval $I$. 
		Define the left (resp. right ) point oscillation functions as
		\[
		\omega_x^{\pm} (\epsilon): = \left\lbrace 
		\begin{array}{ll}
	 		 \sup_{I} | f(x+ \epsilon) - f (x)| & I=[x , x + \epsilon] \\
			 \sup_{I} | f (x) -f(x- \epsilon)| &  I =[ x - \epsilon, x]   
		\end{array} \right. 
		\]
	%The interval $I$ is omitted for clarity of presentation but is given in the context.  
	\end{definition}
	%%%%%%%%%%%%%%%%%%%%%%%
	In such way the directionality information is preserved. 
	It can be established that these quantities are majorized by the oscillation. 
	%%%%%%%%%%%%
	% Prop
	%%%%%%%%%%%%
	\begin{proposition}[Majorization of the point oscillation]\label{prop:oscmajor}
		\[
		\osci{I} f \geq \omega^{\pm}_x (\epsilon) \geq |\deltapm{f}{x}|, \quad I=[x, x \pm \epsilon ]
		\]
		$\omega_x$ is a non-decreasing non-negative function. 
	\end{proposition}
	%%%%%%%%%%%%%%
	\begin{proof}
		Suppose that $f$ is positive in \textit{I}.
		\[
		\osci{I} f= \sup_{ I} f - f(x) + \underbrace{f(x) - \inf_{ I} f}_{\geq 0} \geq \sup_{ I} f - f(x)= \omega^{\pm}_x (\epsilon)
 		\]
 		The second inequality is trivial. 
 		The third assertion follows from the properties of the supremum:
 		For $ A \subset B $ implies 
 		$
 		\sup_{ A}f \leq \sup_{ B} f
 		$. 
	\end{proof}
	%%%%%%%%%%%%%%%%
	
	%%%%%%%%%%%%%%
	%  Corr
	%%%%%%%%%%%%%
	\begin{lemma}[Second Oscillation Lemma]\label{corr:contosc}
		Consider a function \textit{f} continuous in the compact interval $I$ of length \textit{h}.
	 	Then  $ f \cong \predf{C}{I}  \Longleftrightarrow \omega_x (h) \cong \predf{C}{0, h} $ and  $\omega_x (0)=0 $.
		%Then $\omega_x (h) $ is continuous in its domain and $\omega_x (0)=0 $.
	\end{lemma}
	%%%%%%%%%%%%%%%%%%%
	\begin{proof}
		Forward implication:
		The proof follows directly from Prop \ref{prop:oscmajor} by application of the First Oscillation Lemma (\ref{th:osc}).
		Furthermore, since $f(x)=a$ is fixed then if $f$ is continuous so is  $f - a$. 
		
		Reverse implication:
		Consider the right-continuous case. 
		By hypothesis
		\[
		\llim{\epsilon}{0}{\omega^{+}(\epsilon)}= \limsup\limits_{\epsilon \rightarrow 0} | f(x+\epsilon) - f(x) |=0
		\]
		However, 
		\[
		 \limsup\limits_{\epsilon \rightarrow 0} | f(x+\epsilon) - f(x) | \geq 
		  \liminf\limits_{\epsilon \rightarrow 0} | f(x+\epsilon) - f(x) | \Longrightarrow 
		  \liminf\limits_{\epsilon \rightarrow 0} | f(x+\epsilon) - f(x) |=0
		\]
		by majorization. 
		Therefore, both limits coincide and the function $f$ is right-continuous about $x$.
		Let 
		\[
		| \omega^{+}_x (\epsilon) - \omega^{+}_x (\epsilon^\prime) | \leq \mu
		\]
		where $\mu$ is arbitrary but fixed.
		$\sup_{I} f = f(x^\prime)$, $\sup_{I^\prime} f = f(x^{\prime\prime})$, 
		Then for $I^\prime=[x, x + \epsilon^\prime]$
		\[
			| \omega^{+}_x (\epsilon) - \omega^{+}_x (\epsilon^\prime) | = | \sup_{I} f - \sup_{ I^\prime}f |
			=|f(x^\prime) -f(x^{\prime\prime}) | \leq \mu
		\]
		However, $|x^\prime-x^{\prime\prime}| \leq \min (\epsilon, \epsilon^\prime):=\delta$.
		Since $\epsilon$ and $\epsilon^\prime$ can be made arbitrary small then $f$ is continuous in x. 
		The left- continuous case can be proven by the substitution $\epsilon \rightarrow - \epsilon$.
	\end{proof}
	%%%%%%%%%%%%%%%%
	 \begin{corollary}\label{corr:osc}
		The following two statements are equivalent
		\[
		\llim{\epsilon}{0}{	\omega^{\pm}_x(\epsilon)} >0 \Longleftrightarrow \llim{ \epsilon}{0}{ f(x \pm \epsilon) \neq f(x)}
		\]
	\end{corollary}	
	%%%%%%%%%%%%%%%%%

	 %%%%%%%%%%%%%%
	 %  Theorem
	 %%%%%%%%%%%%%%
	 \begin{theorem}[Properties of $\omega_x (\epsilon)$ for sub-additive functions]\label{th:subaddosc}
	
	 	Consider a \textbf{sub-additive} bounded function  \textit{f} on   $I=[x, x + \epsilon]$. 
	 	Consider any two compact nested sub-intervals  $I_a= [x, x+a ) \subseteq \ I_b =[x, x +b ) \subseteq I $, such that $ a+b \leq \epsilon $ and assume  that $\omega_x(a)\neq 0$ and  $\omega_x(b)\neq 0$.
	 	Then the triangle inequality holds: 
	 	\begin{equation}\label{eq:subadd}
	 	\omega_x(a + b ) \leq \omega_x(a)  + \omega_x(b)
	 	\end{equation}
	 	Under the same hypothesis, for a real $\lambda  \geq 1$
		\begin{equation}\label{eq:subadd1}
		\omega_x( \lambda a ) \leq \lambda \; \omega_x(a)   
		\end{equation}
		\begin{equation}\label{eq:subadd2}
		\omega_x ( a+ \lambda b) \leq \omega_x (a) + \lambda \; \omega_x (b)
		\end{equation}
 		For a real $\lambda$, such that $0 \leq \lambda  \leq 1$,
		 \begin{equation}\label{eq:subadd11}
			 \omega_x( \lambda a ) \geq \lambda \; \omega_x(a)   
		 \end{equation}
		 \begin{equation}\label{eq:subadd21}
			 \omega_x ( a+ \lambda b) \geq \omega_x (a) + \lambda \; \omega_x (b)
		 \end{equation}
		 %%%%%%%%%%%%%%%%%%
	Moreover, $\omega_x$ is concave.
%		 \begin{equation}\label{eq:subadd3}
%			 \omega_x (\lambda a+ \left(1-  \lambda \right)  b) \geq \lambda \; \omega_x (a) + \left(  1 - \lambda \right)   \omega_x (b) 
%		 \end{equation}
\end{theorem}
 %%%%%%%%%%%%%%%%%%%%%
	 \begin{proof}
	 	%The first assertion follows from the properties of the supremum. 
	 	Inequality \ref{eq:subadd} follows from Lemma  \ref{th:superaddtive}.
	 	
	 	For the other inequalities the approach is more nuanced. 
	 	Suppose that $f$ is increasing. Then inequalities \ref{eq:subadd1} -- \ref{eq:subadd21} and concavity follow from  Lemma \ref{th:convsuperadd}.
	 	
	 	Suppose that $f$ is non-negative and non-decreasing. Consider the following table of values
	 	\[
	 	\begin{matrix}{}
	 	a^\prime \leq & a \leq & b^\prime < & b  < & c^\prime < & a+b \\
	 	\sup_{ A} f = f(a^\prime) \leq & f(a) \leq & \sup_{ B} f =f (b^\prime) \leq  & f (b) \leq & \sup_{ I} f = f(c^\prime) \leq & f(a+b)
	 	\end{matrix}
	 	\]
	 	Let
	 	\[
	 	M:= \omega_x (a^\prime) + \omega_x (b^\prime) - \omega_x (c^\prime) = 
	 	f(a^\prime) + f(b^\prime) - f(c^\prime)- f(x) 
	 	\]
	 	Then there are two cases to consider:
		Let $f(x) \leq f(a^\prime) \leq f(b^\prime) = f(c^\prime) $.
		Then $M = f(a^\prime) - f(x) \geq 0$. Therefore,  Lemma \ref{th:convsuperadd} can be applied as well.
	 	
	 	Let $f(x) \leq f(a^\prime) = f(b^\prime) \leq f(c^\prime) $.
	 	Then 
	 	\[
	 	M = 2f(a^\prime) - f(c^\prime) - f(x) \leq f(a^\prime) - f(c^\prime)  \leq 0
	 	\]
	 	However, in this case $\omega_x(a)=0$ so it must be excluded for the assertion to hold.
	
	 	Inequality 5 follows from the concavity. 
	\end{proof}
%%%%%%%%%%%%%%%%%%%%%%%%%%

	Conversely, for a \textbf{super-additive} bounded function \textit{f} on \textit{I} an analogous result can be stated.
	%%%%%%%%%%%%%%
	%  Theorem
	%%%%%%%%%%%%%%
	\begin{theorem}[Properties of $\omega_x (\epsilon)$ for super-additive functions]\label{th:subaddosc2}
	Consider a \textbf{super-additive} bounded function  \textit{f} on   $I=[x, x + \epsilon]$. 
	%Then $\omega_x$ is a non-decreasing non-negative function. 
	Consider any two compact nested sub-intervals  $I_a= [x, x+a ) \subseteq \ I_b =[x, x +b ) \subseteq I $, such that $ a+b \leq \epsilon $ and assume  $\omega_x(a)\neq 0$ and  $\omega_x(b)\neq 0$.
	Then
	\begin{equation}\label{eq:subaddb}
	\omega_x(a + b ) \geq \omega_x(a)  + \omega_x(b)
	\end{equation}
	For a real $\lambda  \leq 1$
	\begin{equation}\label{eq:subaddb1}
	\omega_x( \lambda a ) \geq \lambda \; \omega_x(a)   
	\end{equation}
	\begin{equation}\label{eq:subaddb2}
	\omega_x ( a+ \lambda b) \geq \omega_x (a) + \lambda \; \omega_x (b)
	\end{equation}
	For a real $\lambda$, such that $0 \leq \lambda  \leq 1$,
	\begin{equation}\label{eq:subaddb11}
	\omega_x( \lambda a ) \leq \lambda \; \omega_x(a)   
	\end{equation}
	\begin{equation}\label{eq:subaddb21}
	\omega_x ( a+ \lambda b) \leq \omega_x (a) + \lambda \; \omega_x (b)
	\end{equation}
	%%%%%%%%%%%%%%%%%%
	Moreover, $\omega_x$ is convex.
%	\begin{equation}\label{eq:subaddb3}
%	\omega_x (\lambda a+ \left(1-  \lambda \right)  b) \leq \lambda \; \omega_x (a) + \left(  1 - \lambda \right)   \omega_x (b) 
%	\end{equation}
	\end{theorem}
	%%%%%%%%%%%%%%%%%%
	\begin{proof}
		 Inequality \ref{eq:subaddb} follows from Lemma  \ref{th:superaddtive}.
		 Inequalities \ref{eq:subaddb1} -- \ref{eq:subaddb21} and convexity follow from  Lemma \ref{th:convsuperadd}.
	\end{proof}
	%%%%%%%%%%%%%%%
	                                                                         
	The next definition is given by Bartle \cite{Bartle2001}[Part 1, Ch. 7, p. 103].
	%%%%%%%%%%%%%%%%%%%%%%
	%  Def BV
	%%%%%%%%%%%%%%%%%%%%%%
	\begin{definition}[Function of Bounded Variation ]\label{def:BV}
		Consider the interval $I=[a,b]$. 
		A partition of $I$ is a set of $n+1$ numbers $ \partd [I]: = (a < x_1 \ldots x_{n-1} < b ) $.
		The function  \funct{f}{R}{R} is said to be of bounded variation on $I$ if and only if
		there is a constant $M > 0$ such that
		\[
		V_\partd[I] := (\partd ) \sum\limits_{i=1}^{n+1} | f(x_{i} - x_{i-1})| \leq M
		\]
		The total variation of the function is defined as
		$$
		Var(f, I ):= \sup_{\partd} V_\partd [I]
		$$
		where the supremum is taken in all partitions \partd. 
		The class of function of bonded variation in a compact interval $I$ will be denoted as  $BV[I]$. 
	\end{definition}
	%%%%%%%%%%%%%%%%%%%%%%%%

	Under this definition the following proposition can be stated:   
	%%%%%%%%%%%%
	%  Corr
	%%%%%%%%%%%%
	\begin{proposition}
		If $f$ is either monotonously increasing or decreasing on $I=[x, x+ \epsilon]$ then
		$	\omega_x (\epsilon)= |f(x+\epsilon) - f(x)| = |\Delta_{\epsilon}^{+} f(x) |$.
		If $f$ is either monotonously increasing or decreasing on $I=[ x- \epsilon, x]$ then
		$	\omega_x (\epsilon)= |f(x-\epsilon) - f(x)| = |\Delta_{\epsilon}^{-} f (x) |$.
		Under the same hypotheses $	\omega_x \cong BV[I]$.
	\end{proposition}
	%%%%%%%%%%%%%%
	\begin{proof}
	Consider an increasing collection of intervals $ U(n) =\left\lbrace [x, x +a_k ] \right\rbrace^n_{k=1}   $ 
	such that $a_1 < \dots < a_n$.
	Then these form a partition $\partd [x, x+ a_n]$ over $I= \bigcup_{k=1}^n [x, x +a_k ] $. 
	Then $Var[\omega_x, I] = \omega_x(a_n) - \omega_x(a_1)$, which is bounded.
	\end{proof}
	%%%%%%%%%%%%

	The next theorem is a consequence of the  Darboux--Froda theorem. 
	The Darboux--Froda theorem states that the set of discontinuities of a monotone function is at most countable. 
	Hence, by Th. \ref{th:disconect} it is also totally disconnected. 
	In fact, the latter inference can be strengthened to arbitrary functions as Th. \ref{th:discont}.

	%%%%%%%%%%%%%
	%  Theorem
	%%%%%%%%%%%%%
	%%%%%%%%%%%%%%%%%%%%
	\begin{theorem}[Continuity set of oscillation]\label{th:oscont}
	Consider a bounded function $f$ defined on a compact interval $I=[x, x+h]$ and let it be given. 
	The discontinuity set $\Delta_\omega [I]$ of the oscillation function $\omega_x$ is a null set. 
	If $f$ is strictly increasing (respectively decreasing) on $I$
	then the continuity set of the oscillation $\omega_x $ can be written as
	\[
	\mathcal{C}_\omega [I] = \bigcup_{k=1}^\infty (a_k, b_k), \quad  b_k \leq a_{k+1} 
	\]
	\end{theorem}	
    %%%%%%%%%%%%%%%%%%%%
	\begin{proof}
		Consider the interval $I= [x, x+ h]$ with length $h= |I| $  and denote the left-open interval $J_h=(0, h]$ then $\exists q \in \fclass{Q}{} \cap J_h $.
		Therefore, there is a map $J_h \mapsto q$.
		Since $\omega_x (h)$ is non-decreasing, it has only jump discontinuities (since bounded, increasing sequences of numbers have limits).
		Indeed, by the LUB and GLB properties, if $\omega_x$ is (right-)continuous about $h$
		\begin{multline*}
			\left| \sup_{ \epsilon}   \omega_x (h + \epsilon) - L_1 \right|  \leq \mu/2, \quad \epsilon:: \mu \\
			\left| \inf_{ \epsilon}  \omega_x (h + \epsilon) - L_2 \right|  \leq \mu/2 \longrightarrow \\
	 	  		\left| \sup_{ \epsilon} 	  \omega_x (h + \epsilon) - \inf_{ \epsilon} 	\omega_x (h + \epsilon) - L_1 + L_2 \right| =
	 \left|  \omega_x (h + \epsilon) -  	\omega_x (h) -	  L_1 + L_2 \right|  \leq \mu 	 	   
		\end{multline*}
		by the non-decreasing property (Prop. \ref{prop:oscmajor}).  
		On the other hand,
		\[
		\left|  \omega_x (h + \epsilon) -  	\omega_x (h) \right| =
		\left| \sup_{ h+ \epsilon} f - \sup_{ h} f \right| \leq \left| \sup_{ h} f - f(x)\right| =\omega_{x}  (h)  %	\leq 	\left| \sup_{ h+ \epsilon} f\right| 
		\]
	    Therefore, if $h \notin \mathcal{C}_\omega$, that is, if $L_1\neq L_2$ 
	    by the Second Oscillation Lemma  there is some number $x^\prime$, such that 
		\[
		\omega_{x^\prime} ( h)   \geq L_1 - L_2 > 0
		\]
		and there is a jump discontinuity at $h$.
		In such a case, $\exists p \in \fclass{Q}{} \cap [L_1, L_2] $
		so that $[L_1, L_2] = L_p $ can be labelled for a uniquely chosen rational $p$.
		 Let $d_p =|L_2-L_2| $.  
		 If the number $d_p$ is the same for all $x \in I$ then we are done. 
		 On the other hand, suppose that $d_p$ varies in function of $x$.
		Since $\omega_{x}$ is non-decreasing then for another $p^\prime  \neq p\longrightarrow  L_p \cap  L_p^\prime= \emptyset $.
		Therefore, there is an isomorphism $p \longleftrightarrow h$.
		Therefore, the set of continuity of  $\omega_{x}$ is $J_h \setminus \{h\} $ which is an open countable interval.
		 Hence, the second claim follows. 
	\end{proof}
	%%%%%%%%%%%%%%%%%%%%%%%

	Having established these properties, we will characterize the set of discontinuities of a function using the following definition:

    %%%%%%%%%%%%%%%%%
	%   Def
	%%%%%%%%%%%%%%%%
	\begin{definition}\label{def:discont}
		Define the set of discontinuity for the function $F$ in the compact interval $I$ as
		\[
		\Delta [F, I]: = \left\lbrace x:  \osc{F}{x}{}{\pm} >0, \ x \in I \right\rbrace 
		\] 
		or if the context is known $\Delta [F, I] \equiv \Delta [ I] $. 
		In particular, under this definition $\osc{F}{x}{}{}= \infty$ is admissible.
	\end{definition}	
	%%%%%%%%%%%

%	The Darboux--Froda's theorem states that the set of discontinuities of a monotone function is at most countable. 
%	Hence, by Th. \ref{th:disconect} it is also totally disconnected. 
%	In fact, the latter inference can be strengthened to arbitrary functions.
%	The result is given here for completeness of the following discussion. 
	%%%%%%%%%%%%%%
	%  theorem
	%%%%%%%%%%%%%
	\begin{theorem}[Disconnected discontinuity set]\label{th:discont}
	 Consider a bounded function $F$ defined on the compact interval $I$.
	 Then its set of discontinuity $\Delta [F, I]$ is totally disconnected in $I$. 
	\end{theorem}
	%%%%%%%%%%%%%%%%%
	\begin{proof}
		Consider a decreasing collection of closed nested intervals from a partition of the interval $I_1=[x, x+ h]$ ($h>0$, not necessarily small); that is $ \left\lbrace I_k= [x, x +a_k ] \right\rbrace^n_{k=1}   $,
		$I_{k+1} \subset I_{k} \subset \ldots \subset I_1$.
 
		Since the case when the function is locally constant is trivial we consider only two cases: 
		increasing and decreasing.
	
		 Let $E_n = I_n \cap I_{n-1}$ and define
		\begin{flalign*}
		\Delta_n & := \sup_{I_n} F - \inf_{I_n} F \\
		\Delta_{n-1} & := \sup_{I_{n-1}} F - \inf_{I_{n-1}} F \\
		\Delta_E & := \sup_{E_n} F - \inf_{E_n} F \\
		\end{flalign*}
		as indicated in  Fig. \ref{fig:osc} :
		
		%%%%%%%%%%%
		%  Figure
		%%%%%%%%%%%%
		\begin{figure}[hpt]
			\centering
			\includegraphics[width=0.5\textheight]{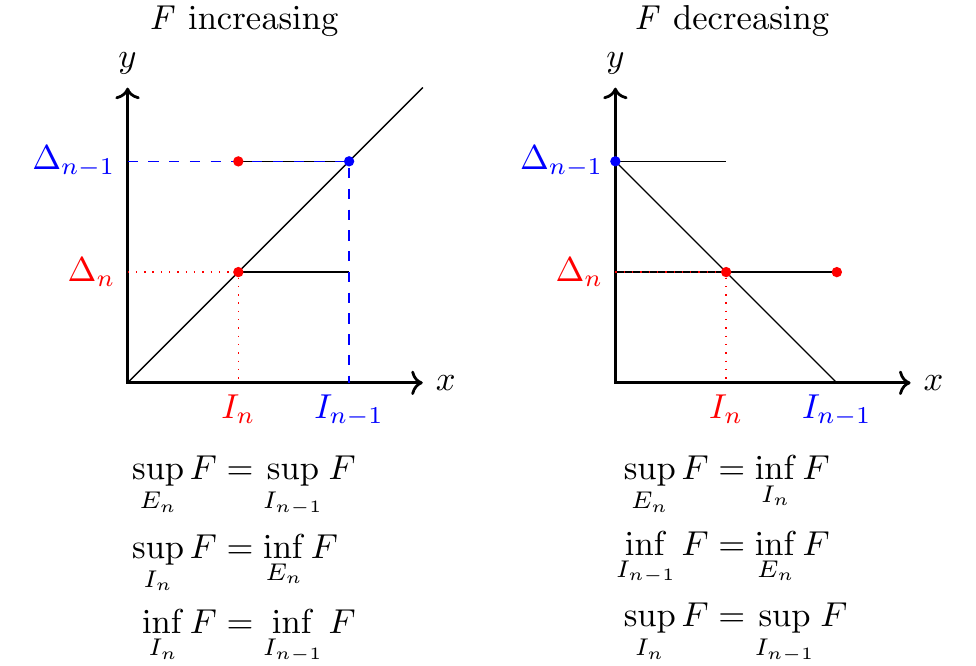}
			\caption{Schematic of the interval construction}\label{fig:osc}
		\end{figure}
		%%%%%%%%%%%%%%%%%%%
	
	%%%%%%%%%%%
  		\begin{description}
  			%%%%%%%
  			\item[Increasing case]  Suppose that $F$ is increasing in $I_{n-1}$ then
  		
  			\begin{multline*}
  			\Delta_{n-1} - \Delta_n =  \sup_{I_{n-1}} F - \inf_{I_{n-1}} F - \sup_{I_{n}} F + \inf_{I_{n}} F =  \\
  			\sup_{E_n} F - \inf_{I_{n-1}} F - \inf_{E_n} F + \inf_{I_{n-1}} F = \Delta_E = 
  			%\omega_{x_n} (|E_n|) 
  			\osci{E_n} f
  			\end{multline*}
  			
  			%%%%%%%%%%%%
  			\item[Decreasing case] Suppose that $F$ is decreasing in $I_{n-1}$ then
  			$$
  			\Delta_{n-1} - \Delta_n = % \sup_{I_{n-1}} F - \inf_{I_{n-1}} F - \sup_{I_{n}} F + \inf_{I_{n}} F = \\
  			\sup_{I_n} F - \inf_{E_{n}} F - \sup_{I_n} F + \sup_{E_{n}} F = \Delta_E = 
  			%\omega_{x_n} (|E_n|)
  			\osci{E_n} f
  			$$	   
  		\end{description}
		%%%%%%%%%%%%%%%

 			On the other hand, $ \osci{I_n}f= \Delta_n$ and $ \osci{I_{n-1}}f = \Delta_{n-1}$.
		  %On the other hand, $ \omega_{x} (|I_n|)= \Delta_n$ and $ \omega_{x} (|I_{n-1}|)= \Delta_{n-1}$.
		  If $F$ is continuous on the opening of $I_{n-1}$, in limit
		  \[
%		  \llim{n}{\infty}{\omega_{x} (|I_n|)- \omega_{x} (|I_{n-1}|) } =0 \Rightarrow  \llim{|E_n|}{0} { \omega_{x_n} (|E_n|)} =0
		  		  \llim{n}{\infty}{  \osci{I_n}f - \osci{I_{n-1}}f } =0 \Rightarrow  \llim{|E_n|}{0} { \osci{E_n}f  } =0
		  \]
		  and $\Delta [F, I] \setminus \{x\} = \emptyset$. 
		  On the other hand, $F$ can be discontinuous on the boundary points $\{x, x+a_{n-1}\}$, which are disconnected.
		  
		  Suppose then that both endpoints are points of discontinuity.
		  Hence, $\{x, x+a_{n-1}\} \subseteq \Delta [F, I]$.
		  If $F$ is discontinuous on the opening of  $I_{n-1}$, we take $I_n$ and proceed in the same way.
		  Therefore,  $\Delta [F, I] \neq \emptyset$ 
		   is a union of totally disconnected sets and hence it is totally disconnected.

	\end{proof}
	%%%%%%%%%%%%%%%%%%

 %%%%%%%%%%%%%%%%
 %   Section
 %%%%%%%%%%%%%%%%%
 \section{Moduli of continuity}\label{sec:contmod}
 
 The moduli of continuity will be discussed as second-order properties of the preimage functions  
 and will be studied in a point-wise manner.
 A modulus will be indexed by a point $x$ in the domain of the preimage function but no other restrictions will be placed there.
 %%%%%%%%%%%%%%%%
 %   Def
 %%%%%%%%%%%%%%%%
 \begin{definition}[Modulus of continuity]
 	A point-wise modulus of continuity $\funct{g_x}{R}{R} $ of a function $\funct{f}{R}{} J \subseteq \fclass{R}{} $  is a 
 	\begin{enumerate}
 		\item non-decreasing continuous function, such that 
 		\item $g_x(0)=0$ and 
 		\item $ | \deltapm{f}{x}  | \leq  K\ g_x (\epsilon)  $ 
 		holds in the interval $I=[x, x \pm \epsilon] \subset J$  for some constant $K$.
 	\end{enumerate}

 	A regular modulus is such that $g_x(1)=1$.
 \end{definition}
 %%%%%%%%%%%%%%
 
 Unsurprisingly, under this definition every continuous function admits a modulus of continuity:
 %%%%%%%%%%%%%%%%
 %  Lemma
 %%%%%%%%%%%%%%%%
 \begin{theorem}[Modulus characterization theorem]\label{th:modcont}
 	Every continuous function admits a modulus of continuity on an interval, which is a subset of its domain. 
 	Any modulus of continuity is $BVC[I]$ for such interval \textit{I}. 
 \end{theorem}
 \begin{proof}
 	Consider the point oscillation function $\omega_x (\epsilon)$.
 	Then $\omega_x$ is non-decreasing.
 	Trivially,  
 	$| \deltapm{f}{x}  | \leq   \omega_x (\epsilon)
 	$ holds.
 	Finally, $\omega_x$ is continuous and $\omega_x(0)= \osc{f}{x}{\epsilon}{+} =0 $ by Prop. \ref{corr:contosc}.
 	Then $g_x(\epsilon)= \omega_x (\epsilon) / \omega_x (1)  $.
 \end{proof}
 %%%%%%%%%%%%
 The point oscillation function used is the proof of Th. \ref{th:modcont} will be called a \textit{canonical modulus of continuity} of a continuous function.

 %%%%%%%%%%%%%%%%
 %  Section
 %%%%%%%%%%%%%%%%	
 \subsection{Classification of the moduli of continuity}
 \label{sec:classmod}    
 
 %%%%%%%%%%%%
 %  Prop
 %%%%%%%%%%%%
 \begin{proposition}\label{prop:classif}
 	Suppose that $\omega_x$ is strictly sub-additive in $I=[0, h]$. 
 	Then $\omega_x^\prime (0) = \infty$.
 
 	Suppose that $\omega_x$ is additive. 
 	Then $\omega_x$ is linear and homogeneous in \textit{h} and $\omega_x^\prime (0) $ exists.

	Suppose that $\omega_x$ is  super-additive.
	Then $\omega_x^\prime (0) $ exists.
 \end{proposition}
 \begin{proof}
 	The proof proceeds in three cases.
 	\begin{description}
 		%%%%%%%%%%%%%
 		% case
 		%%%%%%%%%%%%%
 		\item[Strictly sub-additive case]  
 		Suppose that the derivative exists finitely and let $M > \omega_x^\prime( 0)  \geq m >0$.
 		By sub-additivity there is $h$, such that
 		\[
 		2\, \omega_x (h/2) > \omega_x (h) \Rightarrow  M> \frac{2}{h} \; \omega_x \left( \frac{h}{2} \right)  > \frac{1}{h} \; \omega_x \left( h \right) \geq m
 		\]
 		Then by induction:
 		\[
 		M > \frac{2^n}{h} \; \omega_x \left( \frac{h}{2^n} \right) \geq m \Rightarrow  \frac{M}{2^n} > \frac{1}{h} \; \omega_x \left( \frac{h}{2^n} \right) \geq  \frac{m}{2^n}
 		\]
 		Taking the limit in $n \rightarrow \infty $ leads to
 		\[
 		0 > \omega_x^\prime( 0) \geq 0
 		\]
 		which is a contradiction. Therefore, the limit does not exist finitely and $\omega_x^\prime( 0) = \infty$.

		%%%%%%%%%%%%%%
		% case
		%%%%%%%%%%%%%%
 		\item[Additive case] 
 		By additivity, for all integer $k$ :  $\omega_x ( k h)= k \; \omega_x (h) $. Then by change of variables $z=k h$.
 		$ \omega_x ( z) = k \; \omega_x ( z/h)$. 
 		Therefore,   $\omega_x ( q h) = q \omega_x (h) $ for all rational $q$.
 		Then by continuity, $\omega_x (h)= K h$ for some $K>0$. 
 		
 		%%%%%%%%%%%%%%
 		% case
 		%%%%%%%%%%%%%%
 		\item[Super-additive case]
 		\[
 		2\, \omega_x (h/2) \leq \omega_x (h) \Rightarrow  \frac{2}{h} \; \omega_x \left( \frac{h}{2} \right)  \leq \frac{1}{h} \; \omega_x \left( h \right)  
 		\]
 		Then by induction:
 		\[
 	  0 \leq \frac{2^n}{h} \; \omega_x \left( \frac{h}{2^n} \right) \leq \frac{2^{n-1}}{h} \; \omega_x \left( \frac{h}{2^{n-1}} \right) \leq \ldots     \leq \frac{1}{h} \; \omega_x \left(h \right)  
 		\] 
 		Since the sequence is bounded from below and decreasing it has a limit. 
 		Taking the limit in $n \rightarrow \infty $ leads to
 		\[
 		  \omega_x^\prime( 0) \leq \frac{1}{h}   \omega_x \left(h \right)   
 		\]
 		therefore, the derivative exists finitely at $h=0$.  
 	\end{description}
 	%%%%%%%%%%%%%%%%%%%
 \end{proof}
 %%%%%%%%%%%%%%%%%%%%%%%%
 \begin{example}
 	An illustrative example of the super- additive case is the function $f(x):= e^{-\frac{1}{x^2}}$.
 	For every $n>0$\[
 	\llim{x}{0}{\frac{ e^{-\frac{1}{x^2}}}{x^n}}=0
 	\]
 	The function is super- additive in the interval $x \in \left(-\frac{\sqrt{3}}{2 \sqrt{\log{2}}}, \frac{\sqrt{3}}{2 \sqrt{\log{2}}} \right)$ since there 
 \[2 {{e}^{-\frac{3}{4 {{x}^{2}}}}} < 1\]
 \end{example}
 %%%%%%%%%%%%%%%%%%%%%
 
 %%%%%%%%%%%%%%%%%%%%%%
 Based on this classification result it is useful to apply the following definition.
 %%%%%%%%%%%%%%%%%%
 %  Def
 %%%%%%%%%%%%%%%
 \begin{definition}[g-continuous class]\label{def:grclass}
 	Define the growth class $\fcclass{C}{g}[I]$ induced by the modulus of continuity $g(|I|)$ by the conditions:
 	If $f \cong \fcclass{C}{g}[I] $ on the compact interval $I$ then 
 	\begin{enumerate}
 		\item $C_x= \llim{\epsilon}{0}{\frac{\omega_x (\epsilon)}{ g(\epsilon)}  }  $ exists finitely,
 		\item $C_x$ is non-zero and
 		\item $| \deltapm{f}{x} | \leq C_x \; g (\epsilon) $.
 		 %\item $| \deltapm{f}{x} | \leq C \; g (\epsilon) $, for a constant $C$, such that $C_x \leq C$.
 	\end{enumerate}
 	for $\epsilon= |I|$.
 	To emphasize the dependence on $x$ we may write $\fcclass{C}{g}_x$ and skip $I$ if it is known. \end{definition}
 %%%%%%%%%%%%%%%%
 This definition encompasses the definitions of the H\"older and Lipschitz functions.  
 So that
$
 \fclass{L}{} \equiv \holder{1} \equiv \fcclass{C}{g} $ for $ g (x) =x
 $
 or
$
 \holder{\alpha} \equiv \fcclass{C}{g},  $ for $g (x) =x^\alpha,    0 < \alpha <1
$.
 
 By Prop. \ref{prop:classif} the  modular functions can be classified into three distinct types
 \begin{description}
 	\item[Lipshitz] for which
 	$ \llim{\epsilon}{0}{\frac{\omega_x (\epsilon)}{  \epsilon}  } < L$ for some $L$.
 	These are either linear or otherwise $\omega_x^\prime (0)$ exits finitely and hence they are Lipschitz. 
 	%Therefore, $\omega_x^\prime (0) <L$ by the L'H\^opital's rule and this function is linear by Prop \ref{prop:classif}.
 	\item[Singular] (or strongly non-linear) for which the ratio $\omega_x (\epsilon) /\epsilon$ diverges and
 	$\omega_x^\prime (0) = \infty $ by the L'H\^opital's rule.
 \end{description}
 %%%%%%%%%%%%%%%%%%%%%%

%%%%%%%%%%%%%
%  Section
%%%%%%%%%%%%%
\section{Generalized maximal  $\omega$ derivatives}
\label{sec:genderivomega}

%%%%%%%%%%%%%%%%%
%  definition
%%%%%%%%%%%%%%%%%
\begin{definition}
	For a function $f $ define superior and inferior, and respectively forward and backward, maximal $\omega$ modular derivatives, as the limit numbers $L$
	\begin{flalign*}
	\left| \sup_{ \epsilon} \frac{\deltapm{f}{x} } {\modcont{x}{\epsilon} } - L  \right| < \mu  &\Longrightarrow \bar{\mathcal{D}}^{\pm}_\omega f(x)= L \\
	\left| \inf_{ \epsilon} \frac{\deltapm{f}{x} } {\modcont{x}{\epsilon} } - L  \right| < \mu  &\Longrightarrow\ubar{\mathcal{D}}^{\pm}_\omega f(x)= L
	\end{flalign*}
	for all $\epsilon:: \mu$, $\epsilon >0$.
\end{definition}
%%%%%%%%%%%%%%%%

%%%%%%%%%%%%%%%%%%%%%
\begin{remark}
	These derivative functions obviously generalize the concept of Dini derivatives (Def. \ref{def:dini})
	%, given below for convenience:
%	Define the Dini derivatives as the functions
%	\begin{flalign*}
%	\bar{\mathcal{D}}^{\pm} f (x) & = \limsup\limits_{\epsilon \rightarrow 0} \frac{\deltapm{f}{x}}{\epsilon}  \\
%	\ubar{\mathcal{D}}^{\pm} f (x) & = \liminf\limits_{\epsilon \rightarrow 0} \frac{\deltapm{f}{x}}{\epsilon}
%	\end{flalign*}
%	For the function $f $.
\end{remark}
%%%%%%%%%%%%%%%%%%%

Equipped with the above definition we can state the first existence result:
%%%%%%%%%%%%%%%%%
%  theorem
%%%%%%%%%%%%%%%
\begin{theorem}[Bounded $\omega$-derivatives]\label{th:mondiff}
	For a continuous function the four derivative functions exist as  real numbers.
	Moreover, if $f$ is non-decreasing about $x^{+}$
	\begin{flalign*}
	\bar{\mathcal{D}}^{\pm}_\omega f(x) = 1 \\
	0 \leq  \ubar{\mathcal{D}}^{\pm}_\omega f(x) \leq 1
	\end{flalign*} 
	while if $f$ is non-increasing about $x^{-}$ 
	\begin{flalign*}
	\ubar{\mathcal{D}}^{\pm}_\omega f(x) = - 1 \\
	0 \geq \bar{\mathcal{D}}^{\pm}_\omega f(x) \geq -1
	\end{flalign*} 
\end{theorem}
%%%%%%%%%%%%%%%%
\begin{proof}
	Let $I=  [x, x \pm \epsilon]$ be given and $x$ is fixed but we can vary $\epsilon$. 
	Consider the auxiliary function
	\begin{equation}
	\fracvar {f}{x}{\omega}{\epsilon \pm } := \frac{\deltapm{f}{x}}{\modcont{x}{\epsilon}}
	\end{equation}
	The supremum definitions are restatements of the LUB property for 
	\fracvar {f}{x}{\omega}{\epsilon \pm } in terms of the variable $\epsilon$, while the infimum derivatives are restatements with the GLB property of the reals again for the same variable. 
	Therefore, all four numbers exist for a given argument $x$ and, therefore, under the above hypothesis.
	Moreover, 
	since $ |\deltapm{f}{x}| \leq \modcont{x}{\epsilon} $ then for an non-decreasing function 
	$|\bar{\mathcal{D}}^{\pm}_\omega f(x) | \leq 1 $.
	Therefore, by the supremum property 
	$ \sup_{ \epsilon} \frac{\deltapm{f}{x} } {\modcont{x}{\epsilon}}=1$.
	For a decreasing function $f$ it is sufficient to consider $-f$ and apply the same arguments. 
\end{proof}
%%%%%%%%%%%%%%%
%%%%%%%%%%%
% Corr
%%%%%%%%%%%
\begin{corollary}\label{corr:monotone}
	Suppose that $f$ is monotone and continuous function on a compact interval. 
	If $f$ is increasing in $ [x, x +\epsilon]=I$ then
	$ \bar{\mathcal{D}}^{+}_\omega f(x) = \ubar{\mathcal{D}}^{+}_\omega f(x)  = 1 $.
	If $f$ is decreasing in $ [x, x +\epsilon]=I$ then
	$ \bar{\mathcal{D}}^{+}_\omega f(x) = \ubar{\mathcal{D}}^{+}_\omega f(x)  = - 1 $.
	If $f$ is increasing  in $ [ x -\epsilon, x]=I$ then
	$ \bar{\mathcal{D}}^{-}_\omega f(x) = \ubar{\mathcal{D}}^{-}_\omega f(x)  = 1 $.
	If $f$ is decreasing in $[ x -\epsilon, x]=I$ then
	$ \bar{\mathcal{D}}^{-}_\omega f(x) = \ubar{\mathcal{D}}^{-}_\omega f(x)  = - 1 $.
\end{corollary}
\begin{proof}
	Fix $x$ and consider $I=[x, \pm \epsilon]$.
	The proof follows from the fact that in both cases
	$ |\deltapm{f}{x}| = \modcont{x}{\epsilon} $.
\end{proof}
%%%%%%%%%%%%%%
We can give generalized definition of local differentiability (called $\omega$-differentiability) as follows
%%%%%%%%%%%%%%%%%%
%  Def
%%%%%%%%%%%%%%%%%%
\begin{definition}
	A function $f$ is $\omega$-differentiable at $x$ if at least one of the two limits exist
	\begin{flalign*}
	\left| \frac{\deltapm{f}{x} } {\modcont{x}{\epsilon} } - L  \right| < \mu  &\Longrightarrow \mathcal{D}^{\pm}_\omega f(x)= L 
 	\end{flalign*}
	where the conventions for $L, \mu$ and $\epsilon$ are as above. 
\end{definition}
%%%%%%%%%%%%%%%%%%%%%%%%
Note that the definition only supposes that the one-sided limits of the increments -- $\mathcal{D}^{+}_\omega f(x)$ (respectively $\mathcal{D}^{-}_\omega f(x)$) exist as real numbers.
That is, 
\[
\mathcal{D}^{+}_\omega f(x) \neq \mathcal{D}^{-}_\omega f(x)
\]
is admissible.
This is the minimal statement that can be given for the limit of an increment of a function.
Nevertheless, based on two strong properties -- monotonicity and continuity -- it can be claimed that
\begin{proposition}[Monotone $\omega$-differentiation]\label{th:omegamono}
	If a function $f$ is monotone and continuous in a closed interval $I$ then it is continuously $\omega$-differentiable everywhere in the opening $I^{\circ}$.
\end{proposition}
%%%%%%%%%%%%%%%%%%%
\begin{proof}
	The continuity follows directly from Corr. \ref{corr:monotone}, while the restriction comes from the fact that at the boundary only one of the increments can be defined without further hypothesis for the values of $f$ outside of $I$.  
\end{proof}	
%%%%%%%%%%%%%%%%%%%%%%%%%%
\begin{proposition}[BVC $\omega$-differentiation]
	If the function $f$ is $BVC[I]$  in a closed interval $I$ then it is  $\omega$-differentiable everywhere in the opening $I^{\circ}$.
\end{proposition}
%%%%%%%%%%%%%%%%%%%%%%%%%%
\begin{proof}
	The proof follows from the Jordan theorem, since a BV[I] function can be decomposed into a difference of two non-decreasing functions.
	On the other hand, without further hypotheses we can not claim anything about the eventual equality of 
	$\mathcal{D}^{+}_\omega f(x)$ and $\mathcal{D}^{-}_\omega f(x)$
	since $J_{x} = [x- \epsilon, x] \cap [x, x +\epsilon] = \{x\} $ so that we can form only the trivial map $x \mapsto \{x\}$, which without further restrictions of the domain of \textit{x} (i.e. by means of some topological obstructions) is uncountable.
\end{proof}
%%%%%%%%%%%%%%%%%%%%%

We can further utilize the concept of oscillation to give a concise  
general differentiability condition as
\begin{equation}\label{eq:gendiff}
\llim{\epsilon}{0}{\mathrm{osc}_{\epsilon} \frac{\deltapm{f}{x}}{\modcont{x}{\epsilon}}  } = 0 
\end{equation}
%%%%%%%%%%%%%%%%%%

%%%%%%%%%%%%%%%
%  theorem
%%%%%%%%%%%%%%%
\begin{theorem}[Characterization of $\omega$-derivative]\label{th:omegadiff}
	The following implications hold
	\[
	\bar{\mathcal{D}}^{\pm}_\omega f(x)  = \ubar{\mathcal{D}}^{\pm}_\omega f(x)   = \mathcal{D}^{\pm}_\omega f(x) \Longrightarrow f \cong \mathcal{C}[x^{\pm}]
	\]
	\[
	\llim{\epsilon}{0}{\mathrm{osc}_{\epsilon}\frac{\deltapm{f}{x}}{\modcont{x}{\epsilon}}  } = 0 
	\Longleftrightarrow    \bar{\mathcal{D}}^{\pm}_\omega f(x)  =   \ubar{\mathcal{D}}^{\pm}_\omega f(x)   = \mathcal{D}^{\pm}_\omega f(x)
	\]
	so that if Eq. \ref{eq:gendiff} holds  at \textit{x} then $f$ is  $\omega$-differentiable (and hence continuous) at \textit{x}.
\end{theorem}
%%%%%%%%%%%%%%
\begin{proof}
	\begin{description}
		\item[Continuity implication]  
		Consider the inequality
		\[
		\bar{\mathcal{D}}^{\pm}_\omega f(x) = \ubar{\mathcal{D}}^{\pm}_\omega f(x)  \Longrightarrow \left|  \frac{\deltapm{f}{x} } {\modcont{x}{\epsilon}} - L \right|  \leq \mu/2, \ \epsilon :: \mu
		\]
		so that 
		\begin{flalign*}
		L - \mu/2 \leq \frac{\deltapm{f}{x} } {\modcont{x}{\epsilon}}  & \leq L + \mu/2 \Longrightarrow \\
		\sup_{ \epsilon} \deltapm{f}{x}    &\leq \left( L + \mu/2 \right) \modcont{x}{\epsilon} \\
		\left( L - \mu/2 \right)  \modcont{x}{\epsilon} & \leq \inf_{ \epsilon} \deltapm{f}{x}   \Longrightarrow \\
		\left| \frac{ \sup_{ \epsilon} \deltapm{f}{x}}{\modcont{x}{\epsilon}} - L \right| \leq  \mu/2 \\
		\left| \frac{ \inf_{ \epsilon} \deltapm{f}{x}}{\modcont{x}{\epsilon}} - L \right| \leq  \mu/2 \\
		\end{flalign*}
		Let $\sup_{ \epsilon} \deltapm{f}{x}= M$ and $\inf_{ \epsilon} \deltapm{f}{x} =m $.
		Then by triangle inequality 
		\begin{flalign*}
		\frac{ \sup_{ \epsilon} \deltapm{f}{x}}{\modcont{x}{\epsilon}} -  \frac{ \inf_{ \epsilon} \deltapm{f}{x}}{\modcont{x}{\epsilon}} = \frac{ M- m }{\modcont{x}{\epsilon}}  &\leq \mu \\
		M - m &\leq \mu \; \modcont{x}{\epsilon} \\
		\end{flalign*}
		Therefore, in limit $ M -m \leq 0$, hence $M=m$ and $f$ is continuous. 
		This sequence of operations reminds the fact that real numbers are constructed by a limiting process.
		%%%%%%%%%%%%%%%%%%%%%%%%%%
		
		\item[Forward statement]  
		Suppose that $ \bar{\mathcal{D}}^{\pm}_\omega f(x)= L_1$ and 
		$\ubar{\mathcal{D}}^{\pm}_\omega f(x)= L_2$
		Then by LUB
		\begin{flalign*}
		\left| \sup_{ \epsilon} \frac{\deltapm{f}{x} } {\modcont{x}{\epsilon} } - L_1  \right| \leq \mu/2   \\
		\left| \inf_{ \epsilon} \frac{\deltapm{f}{x} } {\modcont{x}{\epsilon} } - L_2  \right| \leq \mu/2   
		\end{flalign*}
		so that 
		\[
		\left| \sup_{ \epsilon} \frac{\deltapm{f}{x} } {\modcont{x}{\epsilon} } - L_1  \right| +
		\left| \inf_{ \epsilon} \frac{\deltapm{f}{x} } {\modcont{x}{\epsilon} } - L_2  \right| \leq \mu 
		\]
		Then by the triangle inequality
		\[
	\left| L_1 - L_2 \right| \leq	\left| \underbrace{\sup_{ \epsilon} \frac{\deltapm{f}{x} } {\modcont{x}{\epsilon} } -
			\inf_{ \epsilon} \frac{\deltapm{f}{x} } {\modcont{x}{\epsilon} }
		}_{ \mathrm{osc}_{\epsilon} \frac{\deltapm{f}{x}}{\modcont{x}{\epsilon}} } +  L_1- L_2 \right|  \leq \mu
		\] 
		Then in limit 	by lemma. \ref{th:osc}
		\[
		\left|  L_1- L_2   \right| \leq 0 \Longrightarrow L_1 = L_2
		\]
		Further, starting from
		\begin{flalign*}
		&\inf_{ \epsilon} \frac{\deltapm{f}{x} } {\modcont{x}{\epsilon} } \leq    \frac{\deltapm{f}{x} } {\modcont{x}{\epsilon} } -	
		\leq \sup_{ \epsilon} \frac{\deltapm{f}{x} } {\modcont{x}{\epsilon} }  \Longrightarrow\\
		0 &\leq    \frac{\deltapm{f}{x} } {\modcont{x}{\epsilon} } -
		\inf_{ \epsilon} \frac{\deltapm{f}{x} } {\modcont{x}{\epsilon} }
		\leq \sup_{ \epsilon} \frac{\deltapm{f}{x} } {\modcont{x}{\epsilon} } -
		\inf_{ \epsilon} \frac{\deltapm{f}{x} } {\modcont{x}{\epsilon} } = \mathrm{osc}_{\epsilon} \frac{\deltapm{f}{x}}{\modcont{x}{\epsilon}} 
		\end{flalign*}
		Therefore, 
		\begin{flalign*}
		\left| \frac{\deltapm{f}{x} } {\modcont{x}{\epsilon} } -
		\inf_{ \epsilon} \frac{\deltapm{f}{x} } {\modcont{x}{\epsilon} }
		\right| \leq    \mu   \Longrightarrow
		\left| \sup_{ \epsilon} \frac{\deltapm{f}{x} } {\modcont{x}{\epsilon} } -  \frac{\deltapm{f}{x} } {\modcont{x}{\epsilon} } \right|  \leq   \mu
		\end{flalign*}
		Therefore, all three limits coincide.
		
		%%%%%%%%%%%%%%%
		\item[Converse statement]  
		Suppose that 
		\[
		\bar{\mathcal{D}}^{\pm}_\omega f(x) = \ubar{\mathcal{D}}^{\pm}_\omega f(x) = L >0
		\]
		By hypothesis 
		\begin{flalign*}
		\left|\underbrace{ L -
			\inf_{ \epsilon} \frac{\deltapm{f}{x} } {\modcont{x}{\epsilon} } 
		}_{  A_\epsilon } \right| \leq    &\mu /2 \\
		\left|\underbrace{ \sup_{ \epsilon}  \frac{\deltapm{f}{x} } {\modcont{x}{\epsilon} } - L
		}_{  B_\epsilon } \right|  \leq   & \mu /2 \\
		\underbrace{\sup_{ \epsilon} \frac{\deltapm{f}{x} } {\modcont{x}{\epsilon} } -
			\inf_{ \epsilon} \frac{\deltapm{f}{x} } {\modcont{x}{\epsilon} }
		}_{ \mathrm{osc}_{\epsilon} \frac{\deltapm{f}{x}}{\modcont{x}{\epsilon}} } &\leq |A_\epsilon| +|B_\epsilon| \leq \mu 
		\end{flalign*}
		Therefore, in limit
		\[
		0 \leq \llim{\epsilon}{0}{}\mathrm{osc}_{\epsilon} \frac{\deltapm{f}{x}}{\modcont{x}{\epsilon}}     \leq 0
		\]
		so that $ \llim{\epsilon}{0}{\mathrm{osc}_{\epsilon} \frac{\deltapm{f}{x}}{\modcont{x}{\epsilon}}}= 0$.
		%%%%%%%%%%%%%%%%
		
	\end{description}
\end{proof}
%%%%%%%%%%%%%%%%%%

%%%%%%%%%%%%%%
%  Corr
%%%%%%%%%%%%%%
\begin{corollary}[Range of $\mathcal{D}^{\pm}_\omega $ ]
	The range of $\mathcal{D}^{\pm}_\omega $ is given by the discrete set
	\[
	\mathcal{D}^{\pm}_\omega f(x) \subseteq \{-1, 0, +1 \}
	\]
\end{corollary}
%%%%%%%%%%%%%
\begin{proof}
	Let $I=[x, x+ \epsilon]$ be given.
	If $f$ is constant in $I$ trivially $	\mathcal{D}^{\pm}_\omega f(x) =0$.
	If $f$ is increasing in $I$ then $	\mathcal{D}^{\pm}_\omega f(x) =1$ and by duality if
	$f$ is decreasing in $I$ then $	\mathcal{D}^{\pm}_\omega f(x) =-1$. 
\end{proof}
%%%%%%%%%%%%%%%%

The $\omega$ non-differentiability set of a continuous function can be characterized by the following theorem.
%%%%%%%%%%%%%%
%  theorem
%%%%%%%%%%%%%
\begin{theorem}[$\omega$ non-differentiability set]
	Consider the  function $f \cong \mathcal{C}[I]$ on the compact interval \textit{I}. Then the sets
	\[
	\Delta^{+}_\omega [I] := \left\lbrace 
	x:  \bar{\mathcal{D}}^{+}_\omega f(x)  > \ubar{\mathcal{D}}^{+}_\omega f(x)  \right\rbrace \cap I ,
\quad
	\Delta^{-}_\omega [I] := \left\lbrace 
	x:  \bar{\mathcal{D}}^{-}_\omega f(x)  > \ubar{\mathcal{D}}^{-}_\omega f(x)  \right\rbrace \cap I 
	\] 
	are  null sets.
	That is for a continous function the $\omega$ non-differentiability set is null.
\end{theorem}
%%%%%%%%%%%%%%%%%%%%%
\begin{proof}
	Consider the case wherever the right $\omega$-derivative does not exist.
	That is, the defining quotient oscillates without a limit.
	Then  for $0<u, v \leq \delta$ 
	\[\label{ineq:D2}
	\left| \frac{\deltaop {f}{x}{u}{+} }{\modcont{x}{u} } -  \frac{\deltaop {f}{x}{v}{+} }{\modcont{x}{v}} \right|  > \mu \tag{D1}
	\]
	for some $\mu>0$.
	We can consider a variable $\xi \in [x, x+ u] \cap [x, x+ v] = [x, x+ \min (u,v)] =J $.
	There is a rational $r \in \fclass{Q}{} \cap J$.
	Associate $ (r, J) \equiv J_r$ so that $J_r$ can be counted by an enumeration of the rationals and index $\delta :: r$.
	Therefore, the set
	\[
	\Delta_{\omega}: = \bigcup_{k=1}^{\infty} \{z:  \rref{ineq:D2}  \  \mathrm{true},  z \in J_k  \}
	\]
	is countable $\forall \delta >0$. Since $\Delta_{\omega}$ is totally disconnected by Th. \ref{th:discont} 
	we can select $\delta_k= \delta /2^k$ and $J_k \subset J_r$.
	Therefore, 
	\[
	| \Delta_{ \omega}| = \sum\limits_{k=1}^{\infty} |J_k | \leq \sum\limits_{k=1}^{\infty} \frac{\delta}{2^k} = \delta
	\]
	and
	$\Delta_{ \omega}$ is a null set.
	The left derivative case holds by duality. 
\end{proof}
%%%%%%%%%%%%%%%%%%%%%%%

	Note that this is the best possible result for the local-type of derivatives and partially corresponds to the expectation of Ampere.

	%%%%%%%%%
	% input
	%%%%%%%%% 
%	\input{contsection1}

	%%%%%%%%%%%%%%%%
	%   Section
	%%%%%%%%%%%%%%%%%
	\section{Continuity sets of derivatives}\label{sec:contderiv}
	
	%%%%%%%%%%%%%%
	%  Section
	%%%%%%%%%%%%%%
	% \subsection{The Lebesgue monotone differentiation theorem}
	% \label{sec:lebesgue}
	
	In the following we re-state the classical result of the Lebesgue differentiation theorem.
	The poof is given using the machinery of $\omega$-differentiation.
	In the following argument I reserve the term "\textbf{strictly monotone function}" to mean only a strictly increasing or strictly decreasing function in an interval.

	%%%%%%%%%%%%%%%%
	%  theorem
	%%%%%%%%%%%%%%%%
	\begin{theorem}[Lebesgue monotone differentiation theorem]
		Suppose that $f$ is strictly  monotone and continuous in the compact interval $I$. 
		Then $f$ is continuously differentiable almost everywhere.
		The set
		\[
		\Delta_f [I] := \left\lbrace 
		x:  f^\prime_{+} (x) \neq f^\prime_{-} (x) \right\rbrace \cap I 
		\] 
		is a  null set.
	\end{theorem}
	\begin{proof}
		Let $\mathcal{D}_\omega f(x) = L >0 $.
		By Corr. \ref{corr:monotone} for $\epsilon:: \mu$
		\begin{flalign*}
			\left|\underbrace{ L -
				\frac{\deltaplus{f}{x} } {\modcont{x}{\epsilon}^{+} } \frac{\modcont{x}{\epsilon}}{\epsilon}^{+}
			}_{  A_\epsilon } \right| \leq    &\mu /2 \\
			\left|\underbrace{  \frac{\deltamin{f}{x} } {\modcont{x}{\epsilon}^{-} }\frac{\modcont{x}{\epsilon}^{-} }{\epsilon} - L
			}_{  B_\epsilon } \right|  \leq   & \mu /2 \\
			\left| \frac{\deltaplus{f}{x} } {\modcont{x}{\epsilon} } \frac{\modcont{x}{\epsilon}}{\epsilon}^{+} -
			\frac{\deltamin{f}{x} } {\modcont{x}{\epsilon}^{-} } \frac{\modcont{x}{\epsilon}^{-} }{\epsilon}
			\right|  &\leq |A_\epsilon| +|B_\epsilon| \leq \mu \\
			\left| \frac{\modcont{x}{\epsilon}}{\epsilon}^{+} -  \frac{\modcont{x}{\epsilon}^{-} }{\epsilon} \right|  \leq \mu 
		\end{flalign*}
		Therefore, by monotonicity using the original notation
		\[
		\left|  \frac{\deltaplus{f}{x} }{\epsilon} - \frac{\deltamin{f}{x} }{\epsilon} \right| = \left| \frac{\deltaop{f}{x}{\epsilon}{2} }{\epsilon} \right| \leq \mu
		\]
		hence $f^\prime_{+} (x) = f^\prime_{-} (x) $ and $\Delta_f [I] =\emptyset$.
	\end{proof}
	%%%%%%%%%%%%%%%%%%%%%%%%%%%%%%%

	Recall the definitions of nowhere monotone functions:
	%%%%%%%%%%%%%%%%%%%%%%%%%%
	% definition
	%%%%%%%%%%%%%%%%%%%%%%%%%%	 
	\begin{definition}\label{def:monotone}
		A function $f$ is non-decreasing  on $I=[a,b]$ if given any $a<x<y<b$
		\[
		f(y) - f(x) \geq 0
		\]
		and non-increasing on \textit{I} if 
		\[
		f(y) - f(x) \leq 0 .
		\]
		A function, which is neither  non-decreasing nor non-increasing changes direction of growth in $I$.
		A function is nowhere monotone (NM[I]) if given any $a<x<y<z<b$
		\[
		\left( f(y) - f(x)\right) \left( f(z) - f(y)\right) \leq 0 
		\]
		so that NM[I] function is  neither  non-decreasing nor non-increasing  on any sub-interval of \textit{I}. 
		A function, which is nowhere monotone at a point (NM[y]), is treated as above while $y$ is fixed.
	\end{definition}
	%%%%%%%%%%%%%%%%%%%%%%%%%

	From the Lebesgue monotone differentiation theorem it follows  that a nowhere differentiable function on an  open interval \textit{I} is simultaneously nowhere monotone on \textit{I}.
	Brown et al. establish that no continuous function of bounded variation BVC is MN[y]. \cite{Brown1999}[Th. 12. Corr. 3].
	That is to say $NM[x]$ for $x \in I$ as above.
	Therefore, it is of interest to establish the following result.

	%%%%%%%%%%%%%%
	%  theorem
	%%%%%%%%%%%%%
	\begin{theorem}[NM continuous $\omega$-differentiability]
		Suppose that $f \cong \mathcal{C}[I]$ and $f \cong NM[I]$. 
		Then   
		$$\mathcal{D}^{\pm}_\omega f(x) \cong \mathcal{C}[I] \Longrightarrow 	\mathcal{D}^{\pm}_\omega f(x) =0$$
	\end{theorem}
	\begin{proof}
		The set 
		$
		\{x: \mathcal{D}^{\pm}_\omega f(x)=1  \}
		$ is totally disconnected. 
		By duality, the set $
		\{x: \mathcal{D}^{\pm}_\omega f(x)=-1  \}
		$ is also totally disconnected.
		Hence, only the set $\{x: \mathcal{D}^{\pm}_\omega f(x)=0  \}$ has connected components.
	\end{proof}
	%%%%%%%%%%%%%%%%%%
	
	%%%%%%%%%%%%%%%
	%  theorem
	%%%%%%%%%%%%%%%
	\begin{theorem}[Continuity of derivatives]\label{th:contderiv}
		Consider a bounded and continuous function $f$ on a compact interval $I$.
		Suppose that $f^\prime_{+} (x)$ and $f^\prime_{-} (x)$ are separately continuous
		then the following holds:
		\begin{enumerate}
			\item $
			f^\prime_{+} (x)= f^\prime_{-} (x) = f^\prime (x)
			$
			\item 	$\Delta_{f, I } := \{ x: f^\prime \notin \mathcal{C},  x \in I  \}$ is totally disconnected with empty interior. % and $F_\sigma$.
			\item The total discontinuity set can be written as 
			$
			\Delta_{f, I} =\Delta_{1, f} \cup \Delta_{2, f}
			$, where $\Delta_{1, f}$ is $F_\sigma$ and $\Delta_{2, f}$ is a null set.
			%	 		\item  The continuity set can be written as	$
			%	 		\mathcal{C} _f= \bigcup_{k=1}^\infty (a_k, b_k), \ b_k \leq a_{k+1}
			%	 		$ and is thus $G_\delta$.
			\item  The continuity set $\mathcal{C} _f$ is $G_\delta$.
		\end{enumerate}
		%%%%%%%%%%%%%%%%%%%%%%
		
	\end{theorem}
	%%%%%%%%%%%%%%%%
	%%%%%%%%%%%%%%%%
	\begin{proof}
		Consider the interval $I= [u, v]$.
		Then there is rational $ r  \in \fclass{Q}{} \cap I$.
		
		Associate $ (r, I) \equiv I_r$ so that $I_r$ can be counted by an enumeration of the rationals.
		
		Assume that $f^\prime_{+} (x)$ and $f^\prime_{-} (x)$ are separately continuous on the  opening
		of $I^\circ_r = I_r -\{u\} - \{v \} $.
		Fix $x$, such that $u \geq x > v$.
		%%%%%%%%%%%%
		\begin{flalign*}
			u > v &  \ \ \  u \geq x > v  \\
			\frac{f (u)- f(v)}{u-v} & = \frac{f (u)- f(x) +f(x) - f(v)}{u-v} = \\
			& \frac{f (u)- f(x)}{u-x} \underbrace{\frac{u-x}{u-v}}_{1- \lambda} +  \frac{f (x)- f(v)}{x-v} \underbrace{\frac{x-v}{u-v}}_{\lambda} = \\
			& \frac{f (u)- f(x)}{u-x} (1- \lambda) +  \frac{f (x)- f(v)}{x-v} \; \lambda \\
			& \ \ \ \ \downarrow \lim\limits_{u \rightarrow x} \ \ \ \ \ \ \ \ \ \ \ \ \ \ \ \ \ \ \ \ \ \downarrow \lim\limits_{v \rightarrow x} \\
			& (1-\lambda )f^\prime_{+} (x) + \lambda f^\prime_{-} (x) = f^\prime_{+} (x) - \lambda \left( f^\prime_{+} (x)  - f^\prime_{-} (x) \right) 
		\end{flalign*}
		By continuity 
		%\begin{equation}
		$$
		\llim{v}{x}{} f^\prime_{+} (v) = f^\prime_{+} (x) = f^\prime_{+} (x) - \lambda \left( f^\prime_{+} (x)  - f^\prime_{-} (x) \right) 
		$$
		%\end{equation}
		However, since $x$ and hence $\lambda \neq 0 $ is arbitrary 
		$f^\prime_{+} (x)  = f^\prime_{-} (x) $ must hold $\forall x  \in I^\circ_r$. 
		Hence, $f^\prime$ is continuous on $I^\circ_r$.
		
		By this argument we establish that
		the set $\Delta_{1, f} := \{ x: f^\prime \not\cong \mathcal{C} \} \cap I$ is  $F_\sigma$,
		where we also assume that whenever $f^\prime (x)$ does not exist it is replaced by a value that makes  
		$f^\prime $ discontinuous. 
		By Th. \ref{th:discont}  the discontinuity set is totally disconnected and with empty interior.  
		
		Let us further consider the case wherever left and right derivatives do not exist (either diverge or oscillate without a limit).
		It is enough to consider the right derivative. 
		Then we have that for $0<u, v \leq \delta$ 
		\[
		\label{ineq:D1}
		\left| \frac{\deltaop {f}{x}{u}{+} }{u} -  \frac{\deltaop {f}{x}{v}{+} }{v} \right|  > \epsilon >0 \tag{D2}
		\]
		for some $\epsilon$.
		We can consider a variable $\xi \in [x, x+ u] \cap [x, x+ v] = [x, x+ \min (u,v)] =J $.
		There is a rational $r \in \fclass{Q}{} \cap J$.
		Associate $ (r, J) \equiv J_r$ so that $J_r$ can be counted by an enumeration of the rationals and index $\delta :: r$.
		Therefore, the set
		\[
		\Delta_{2, f}: = \bigcup_{k=1}^{\infty} \{z:  \rref{ineq:D1} \ \mathrm{true},  z \in J_k  \}
		\]
		is countable  $\forall \delta > 0$. Since it is totally disconnected by Th. \ref{th:discont} 
		we can select $\delta_k= \delta /2^k$.
		Therefore, 
		\[
		| \Delta_{2, f}| = \sum\limits_{k=1}^{\infty} |J_k | \leq \sum\limits_{k=1}^{\infty} \frac{\delta}{2^k} = \delta
		\]
		and
		$\Delta_{2, f}$ is a null set.
		
		The same argument can be applied to the left derivative considering $f(-x)$.
		
		The total discontinuity set can be written as 
		\[
		\Delta_{f, I} =\Delta_{1, f} \cup \Delta_{2, f}
		\]
		Therefore, the continuity set can be written as
		$
		\mathcal{C} _f =  \left( \Delta_{1, f} \cup \Delta_{2, f}\right)^c
		$
		hence it is $G_\delta$.
	\end{proof}
	%%%%%%%%%%%%%%%%%%%%%%

	 %%%%%%%%%%%%%%%%
	 %   Section
	 %%%%%%%%%%%%%%%%%
	 \section{Modular derivatives} 	
	 \label{sec:genderiv}
 
	 As indicated in Sec \ref{sec:intro}, the derivatives can be generalized in different directions. 
     If locality is the leading requirement, then the most natural way for such generalization is to replace the assumption of local Lipschitz growth with the more general modular-bound growth. 
	 In such way one can generalize, previously introduced  fractional velocity of Cherbit \cite{Cherbit1991}. 
 
	% For convenience the reader is recalled first with some useful  definitions.
 
	 %%%%%%%%%%%%%%%%%%%%%5
	 %	definition
	 %%%%%%%%%%%%%%%%%%%%%
	 \begin{definition}
	 	\label{def:fracvar}
	 	Define \textit{g-variation} operators  as
	 	\begin{align}
	 	\label{eq:fracvar1}
	 	\gvarpm {f}{x} := \frac{ \deltapm{f}{x}  }{ g(\epsilon)} 
	 	\end{align}
	 	for a positive  $\epsilon$ and a modular function $g$.
	 \end{definition}
	 %%%%%%%%%%%%%%%%%%%%%%%%%%
	 %%%%%%%%%%%%%%%%%%%%%%%%%%%%%%%%% 
	 \begin{condition}[Modulus-bound growth condition]
	 	For  given  $x$ and a modular function $g$.
	 	\begin{equation}\label{C1} 
	 	\mathrm{osc}_{\epsilon }^{\pm} f (x)  \leq C g \left( \epsilon \right)  \tag{C1}
	 	\end{equation}
	 	for some  $C \geq 0$ and $\epsilon > 0$.    
	 \end{condition}
	 %%%%%%%%%%%%%%%%%%%%%%%%%%%%%
	 \begin{condition}[Vanishing oscillation condition]
	 	For  given  $x$ and $\epsilon>0$ 	
	 	\begin{equation}\label{C2}
	 	\mathrm{osc}^{\pm} \gvarpm {f}{x} =0  \tag{C2}
	 	\end{equation}
	 	where the limit is taken in  $\epsilon$.
	 \end{condition}
	 %%%%%%%%%%%%%%%%%%%%
	 
	 Define the modular derivative as:
	 %%%%%%%%%%%%%%%%%%
	 %  Def
	 %%%%%%%%%%%%%%%
	 \begin{definition}[Modular derivative, g-derivative]\label{def:genderiv}
	 	Consider an interval $[x, x \pm \epsilon]$ and define 
	 \begin{equation}\label{eq:gderiv}
			\gdiffpm{f}{x}   := \llim{\epsilon}{0}{ } \frac{\deltapm{f}{x}}{ g( \epsilon)}
	 \end{equation}	 
	 for a modulus of continuity $g(\epsilon)$.
	 The last limit will be called modular derivative or a g-derivative.
	 
	 NB! We do not demand equality of $ \mathcal{D}^{+}_{ g}  f(x) $ and $ \mathcal{D}^{-}_{ g}  f(x) $.
	 \end{definition}
	 %%%%%%%%%%%%%%%%%%%
	 
      We are ready to establish the existence conditions of the g-derivative.
  
 %%%%%%%%%%%%%%%%%
 %  theorem
 %%%%%%%%%%%%%%%%%%
 \begin{theorem}[Conditions for existence of g-derivative]\label{th:aexit}
 	If \gdiffplus {f}{x}  exists (finitely), then $f$ is right-continuous at $x$ and \ref{C1} holds, and the analogous result holds for \gdiffmin {f}{x}  and left-continuity.
 	
 	Conversely, if \ref{C2} holds then $\gdiffpm {f}{x} $ exists finitely.
 	Moreover, \ref{C2} implies \ref{C1}. 
 \end{theorem}
 %%%%%%%%%%%%%%%%%%%%%%%%%%
 %%%%%%%%%%%%%%%%%%%%%%%%%%%%%%%
 \begin{proof}  
 	We will first prove the case for right continuity.
 	Condition C1 trivially implies the g- continuity, which according to our notation is given  as  $\gvarpm {f}{x} \leq C g(\epsilon)$.   
 	%%%%%%%%%%%%%%
 	%   cases
 	%%%%%%%%%%%%%%%
 	\begin{description}
 		\item[Forward statement]
 		~\\
 		Without loss of generality suppose that $L>0$ is the value of the limit. Then by hypothesis  
 		\[
 		\left| 	\frac{\deltaplus{f}{x}}{g(\epsilon)} - L \right|  < \mu
 		\]
 		holds for every $\mu:: \delta, \epsilon < \delta$ .
 		Straightforward rearrangement gives
 		\[
 		\left| 	f(x+ \epsilon) - f(x) - L g(\epsilon) \right|  < \mu\; g(\epsilon) 
 		\]
 		Then by the reverse triangle inequality
 		\[
 		\left| 	f(x+ \epsilon) - f(x)   \right| - L g(\epsilon) \leq
 		\left| 	f(x+ \epsilon) - f(x) - L g(\epsilon) \right|  < \mu \, g(\epsilon) 
 		\]
 		so that 
 		$
 		\left| 	f(x+ \epsilon) - f(x)   \right| < \left( \mu + L  \right)  g(\epsilon)
 		$.
% 		\footnote {Alternatively, we can also assign a Cauchy sequence to $\delta$ and demand that RHS approaches arbitrary close to 0 implying also $	\mathrm{osc^{+}} [f] (x) =0$.}
 		Further,  by the least-upper-bound property there exists a number $C \leq \mu + L $, such that
 		\[
 		\left| 	f(x+ \epsilon) - f(x)   \right| \leq C g(\epsilon),
 		\]
 		which is precisely the Modulus bound growth condition. 
 		The left continuity can be proven in the same way.
 		%%%%%%%%%%%%%%%%%%%%%%%%%%%%%%%%%%%
 		\item[Converse statement]
 		~\\
 		In order to prove the converse statement we can observe that  condition \ref{C2} implies that
 		$
 		\mathrm{osc}^{+} \gvarplus {f}{x} =0 
 		$
 		so that 
 		\[
 		\mathrm{osc}^{+}_{\epsilon }  \frac{\deltaplus{f}{x}     }{g(\epsilon)} \leq \mu
 		\]
 		for $ \mu::\epsilon$ (and in particular for a Cauchy null-sequence $\mu$)
 		so that 
 		\[
 		\left|\sup_\epsilon \frac{\deltaplus{f}{x} }{g(\epsilon)} - \inf_\epsilon   \frac{\deltaplus{f}{x} }{g(\epsilon)} \right|  \leq \mu
 		\]
 		by lemma \ref{th:osc} and 
 		\[
 		\sup_\epsilon \frac{\deltaplus{f}{x} }{g(\epsilon)} \leq \mu + \inf_\epsilon   \frac{\deltaplus{f}{x} }{g(\epsilon)},
 		\]
 		so that taking the limits in  $\mu$ (and hence $\epsilon$) implies 
 		\[
 		\limsup\limits_{\epsilon \rightarrow 0 }  \frac{\deltaplus{f}{x} }{g(\epsilon)}    = \liminf\limits_{\epsilon \rightarrow 0 }    \frac{\deltaplus{f}{x} }{g(\epsilon)a}  
 		\]
 		Hence $\llim{\epsilon}{0}{ \gvarplus {f}{x} } = L =\gdiffplus{f}{x}$
 		for some real number $L$.

 		However, the latter limit can be rewritten from its definition as
 		\[
 		\left| \frac{\Delta_\epsilon^{+} f (x) - L g(\epsilon)}{g(\epsilon)} \right| < \mu  
 		\]
 		for an arbitrary $\mu :: \epsilon$.
 		Then since $\mu$ is arbitrary by the least upper bound property there is
 		$\epsilon^\prime$, such that
 		\[
 		\left| \Delta_{\epsilon^\prime}^{+} f (x) \right| = \osc{f}{x}{\epsilon^\prime}{+} \leq (\mu +L ) g(\epsilon^\prime)  
 		\]   
 		for $\mu::\epsilon^\prime$  
 		and we identify condition \ref{C1}. 
 	\end{description}
 	%%%%%%%%%%%%%%%%%%%
 	
 	The left case follows by applying the right case, just proved, to the reflected function $f(-x)$.
 \end{proof}
 %%%%%%%%%%%%%%%%%%%%%%%%%%%%%%%%%%%%%%%%%%%%%% 

%%%%%%%%%%%%%%%
%  Section
%%%%%%%%%%%%%%
\subsection{Generalized Taylor-Lagrange property}\label{sec:gtaylor2}

%%%%%%%%%%%%%%%%%%%%%%%%%%
% theorem
%%%%%%%%%%%%%%%%%%%%%%%%%%
\begin{proposition}[Generalized Taylor-Lagrange property]
	\label{th:holcomp1}
	The existence of $\gdiffpm{f}{x}\neq 0$   implies that
	\begin{equation}\label{eq:frtaylorlag}
	f(x \pm \epsilon) = f(x) \pm \gdiffpm {f}{x}   g(\epsilon) + \bigoh{g({\epsilon} ) }   
	\end{equation}
	for the modular function $g$.
	While if  
	\[
	f(x \pm \epsilon)= f(x) \pm K g(\epsilon) +\gamma_\epsilon \; g( \epsilon)  
	\] 
	uniformly in  the interval $ x \in [x, x+ \epsilon]$ for some Cauchy sequence $\gamma_\epsilon = \bigohx{x}$ and $K \neq 0$ is constant in $\epsilon$  then $\gdiffpm {f}{x}  = K$.
\end{proposition}
%%%%%%%%%%%%%%%%%%%%%%%%%%%%%%%%%%%%%%%%%%%%%
\begin{proof}
	We prove only for the forward modular derivative. 	
	The case for the backward modular derivative is proven in the same way following a reflection of the function argument \textit{x}.
	\begin{description}
		\item[Forward statement]  
		By the definition of the modular derivative $\exists \gamma$, such that
		$
		f(x +\epsilon)= f(x) + \gdiffplus {f}{x}  g(\epsilon) +\gamma   
		$.
		Moreover,  $\gamma = \bigoh{g(\epsilon) }  $.
		
		\item[Converse statement]
		Suppose that 
		\[
		f(x +\epsilon)= f(x) + K g(\epsilon) +\gamma_\epsilon \; g(\epsilon)
		\] 
		uniformly in  the interval $ x \in [x, x+ \epsilon]$ for some number $K$ and  $\gamma_\epsilon = \bigohx{x}$.
		Then this fulfils both Modulus bound growth and Vanishing oscillation conditions. 
		Therefore, 	$ K = \gdiffplus{f}{x}$ observing that \llim{\epsilon}{0}{ \gamma_\epsilon}=0.
	\end{description}
\end{proof}
%%%%%%%%%%%%%%%%%%%%%%%%%%%%%%%%%%%%%%%%%%%%%%%%%	 

%%%%%%%%%%%%%%%
%  Section
%%%%%%%%%%%%%%%
\subsection{Characterization by $\omega$-derivatives}\label{sec:ommegachar}
%%%%%%%%%%%%%%%
%  Prop
%%%%%%%%%%%%%%%
\begin{proposition}\label{prop:omegaequiv}
	Consider the modular function $g$. Then
	\[
	\wdiffpm{f}{x} = K \gdiffpm{f}{x}
	\]
	for some constant $K$ wherever all limits exist.
\end{proposition}
%%%%%%%%%%%%%%%%%%%%%%%%%
\begin{proof}
	Let $$K = \llim{\epsilon}{0}{\frac{\omega_{x} ( \epsilon)}{g (\epsilon)}}$$ and suppose that the limit exists as a finite number.
	Then
	\[
	\frac{\deltaplus{f}{x}}{g (\epsilon)} = \frac{\deltaplus{f}{x}}{\modcont{x}{\epsilon }} \frac{\modcont{x}{\epsilon }}{g (\epsilon)}
	\]
	Therefore, the statement of the result follows.
\end{proof}
%%%%%%%%%%%%%%%%%%%%%%%%
In view of Prop. \ref{th:omegamono} this means that a function can change its modulus of continuity point-wise. 
Since the cases of H\"older and Lipschitz functions have been treated extensively in literature we will consider only the general case.

%%%%%%%%%%%%%%%%%
% Section
%%%%%%%%%%%%%%%%
\subsection{Continuity of $g$-derivatives}\label{sec:gderivcont}

%%%%%%%%%%%%%%% 
Gleyzal \cite{Gleyzal1941} established that a function is Baire class I if and only if it is the limit of an interval function.
Therefore, \gdiffpm {f}{x}  are Baire class I from which it follows that \gdiffpm {f}{x} must be continuous on a dense set. 
Moreover, since the continuity set of a function is a $G_{\delta}$ set, (i.e. an intersection of at most countably many open sets), 
from the Osgood-Baire Category theorem  it follows that the set of points of discontinuity of
\gdiffpm {f}{x} is $F_\sigma$ meagre (i.e. a union of at most countably many nowhere dense sets or else  it has empty interior).

Since in the previous sections it was established that the modulus of continuity can be conveniently classified as used conventionally in applied literature we are ready to state an important result concerning the continuity of $g$-derivatives.
First, we have the following theorem:

%%%%%%%%%%%%%%
%  theorem
%%%%%%%%%%%%%%
\begin{theorem}[Continuity of $g$-derivatives]\label{th:contgderiv}
	Suppose that $g$ is a \textbf{strictly sub-additive} modular function on the compact interval I.
	Then wherever \gdiffpm{f}{x} is continuous it is zero. 
	That is 
	\[
	\gdiffpm{f}{x} \cong \mathcal{C}[I] \Rightarrow \gdiffpm{f}{x}=0
	\]
\end{theorem}
%%%%%%%%%%%%%%%%%
\begin{proof}
	Let $\gdiffplus{f}{x}=K>0$. 
	  \begin{flalign*}
	    \frac{\deltaplus{f}{x}}{g(\epsilon) }  & = \frac{ f (x+ \epsilon) - f(x + \epsilon/2) }{ g(\epsilon)} + 
	    \frac{ f (x ) - f(x - \epsilon/2) }{ g(\epsilon)}= \\
	    & = \frac{ f (x+ \epsilon) - f(x + \epsilon/2) }{ g(\epsilon/2)} \frac{g(\epsilon/2)}{g(\epsilon)} + 
	    \frac{ f (x ) - f(x - \epsilon/2) }{g(\epsilon/2) } \frac{g(\epsilon/2)}{g(\epsilon)}
	  \end{flalign*}
	  Therefore, in  limit supremum and by hypothesis of continuity
	  \[
	  K=  K \underbrace{ \limsup\limits_{ \epsilon \rightarrow 0} \frac{ 2\, g(\epsilon/2)}{g(\epsilon)}}_G
	  \]
	  By strict sub-additivity $2 g(\epsilon/2) / g(\epsilon) < 1$,  therefore, the limit $G$ exists.
	  So it is established that
	  $ K = G K < K $, which is a contradiction.
	  Therefore, $K=0$ on the first place. 
	  The case for the left derivative  follows by duality.  
\end{proof}
%%%%%%%%%%%

%%%%%%%%%%%%%%%%
\begin{corollary}
	The continuity requirement is equivalent to requiring that
	$$\llim{\epsilon}{0}{  \frac{ 2 \, g^\prime (\epsilon/2 ) }{g^\prime (\epsilon)}} = 1$$
\end{corollary}
 %%%%%%%%%%%%%%%%%%
 
  %%%%%%%%%%%%%%%
 %  theorem
 %%%%%%%%%%%%%%%
 \begin{theorem}\label{th:mesgderiv}
 	Consider a function $f$ having a strictly sub-additive modulus function $g$ on the compact interval \textit{I}. 
 	Then the set 
 	\[
 	\chi_g^{\pm} (f):= \{x:  \gdiffpm{f}{x} \neq 0  \} \cap I
 	\] 
 	is totally disconnected and of measure zero, that is
 	$|\chi_g^{\pm} (f)| = 0$.
 	The set $\chi_g^{\pm} $ will be called the set of change of $f$.
 \end{theorem}
 %%%%%%%%%%%%%%%%%%%%%%
 \begin{proof}
 	Using the same argument as in the proof of Th. \ref{th:contgderiv} we establish that either $K=0$ allowing for continuity of \gdiffpm{f}{x} or
 	$K \neq 0$ but then \gdiffpm{f}{x} can not be continuous.
 	Furthermore, by Th. \ref{th:contderiv} it follows that $|\chi_g (f)| = 0$. 
 \end{proof}
 %%%%%%%%%%%%%%
 \begin{corollary}
 	Under the same notation, let $g(\epsilon) =\epsilon^\beta$, for $\beta \in (0, 1]$.
 	If $|\chi_g (f)|>0$  then $\beta=1$ and $f$ is Lipschitz. 
 \end{corollary}
 %%%%%%%%%%%%%%%%
 
 %%%%%%%%%%%%%%%%
 \begin{corollary}
 	Under the same hypotheses the image set $\mathcal{D}_g^{\pm}f$ is totally disconnected.
 \end{corollary}
 %%%%%%%%%%%%%%%

 %%%%%%%%%%%%
 %  Section
 %%%%%%%%%%%%%%
 \section{Discussion}\label{sec:disc}
 
   	The relaxation of the differentiability assumption opens new avenues in describing non-linear physical phenomena, for example, using \textit{stochastic calculus}   or the \textit{scale relativity theory} developed by  Nottale \cite{Nottale1989}, which assume fractal character of the space-time geodesics and hence of quantum-mechanical paths.  
   	
 	In contrast to the Riemann-Liouville or Caputo fractional derivatives, the geometrical, and hence physical, interpretation of a modular derivative is easier to establish due to its local character and the demonstrated generalized Taylor-Lagrange property.  
 	That is, presented results demonstrate that the modular derivative provides the best possible local non-linear approximation for its natural modulus of continuity function at the point of interest.
 
 	The desirable properties of the derivatives, such as their continuity,  are established from the more general perspective of the moduli of continuity.   		
 	From the perspective of approximation, derivatives can be viewed as mathematical idealizations of the linear growth. 
 	The linear growth, i.e. the Lipschitz  condition, has special properties, which make it preferred. 
 	Importantly, the statements of the Th. \ref{th:contgderiv} 	and \ref{th:mesgderiv}  give further insight on why the ordinary derivatives are so useful for describing physical phenomena in terms of differential equations.

  \appendix

%%%%%%%%%%%%%%%%%%
% Section
%%%%%%%%%%%%%%%%%%% 
\section{General definitions and conventions}
\label{sec:definitions}

The term \textit{variable} denotes an indefinite number taken from the real numbers. 
Sets are denoted by capital letters, while variables taking values in sets are denoted by lowercase.

%The term \textit{function}  denotes a mapping from one number to another
The action of the function is denoted as $f(x)=y$. 
Implicitly the mapping acts on the real numbers: \funct{f}{R}{R}.	 
If a statement of a function \textit{f} fulfils a certain predicate with argument $A$ (i.e. $Pred[A]$) the following short-hand notation will be used   $f \cong Pred[A] $.

%	 The co-domain, or range, of the function $f: X \mapsto Y$ is denoted as $f[X]=Y$. 
%	 The term \textit{operator} denotes the mapping from a functional expression to functional expression.
%	 The term \textit{functional} denotes the mapping from a functional expression to a number. 
Square brackets are used for the arguments of operators, 
while round brackets are used for the arguments of functions.	
The term Cauchy sequence will always be interpreted  as a null sequence.

Everywhere, $\epsilon$ will be considered as a small positive variable.
%%%%%%%%%%%%%
% definition
%%%%%%%%%%%%%
\begin{definition}[Asymptotic small $\smallO$ notation]
	\label{def:bigoh}
	The notation  $\bigoh{x^\alpha}$ is interpreted as the convention  
	$$
	\llim{x}{0}{ \frac{\bigoh{x^\alpha}}{x^\alpha} } =0 
	$$
	for $\alpha >0 $.
	Or in general terms
	\[
	\bigoh{ g(x)} \Rightarrow 	\llim{x}{0}{ \frac{\bigoh{ g(x)}}{g(x)} } =0 
	\]
	for a decreasing function $g$ on a right-open interval containing 0.
	The notation $\bigohone{x}$ will be interpreted to indicate a Cauchy-null sequence possibly indexed by the variable $x$.
\end{definition}
%%%%%%%%%%%%%%%	

%%%%%%%%%%%%%%%%%%%%%%%%%%
% definition
%%%%%%%%%%%%%%%%%%%%%%%%%%	 
\begin{definition}\label{def:deltas}
	Define the parametrized difference operators acting on the function $f(x)$ as
	\begin{flalign*}
	\deltaplus{f}{x} & :=  f(x + \epsilon) - f(x),\\
	\deltamin{f}{x}  & :=  f(x) - f(x - \epsilon)  
	\end{flalign*}
	for  the variable $\epsilon>0$. 
	The two operators are referred to as \textit{forward difference} and
	\textit{backward difference}
	operators, respectively.
\end{definition}
%%%%%%%%%%%%%%%%%%%%%%
%%%%%%%%%%%%%%%%%%%%%%
%  definition
%%%%%%%%%%%%%%%%%%%%%%

\begin{definition}[Anonymous function notation]
	\label{def:pair}
	The notation for the pair $\mu :: \epsilon$ will be interpreted as the implication that 
	if Left-Hand Side (LHS)  is fixed then the Right-Hand Side (RHS) is fixed by the value chosen on the left, 
	i.e. as an anonymous functional dependency
	$\epsilon= \epsilon (\mu) $.
\end{definition}
%%%%%%%%%%%%%%%%%%%%
 
 %%%%%%%%%%%%%%
 %  Def
 %%%%%%%%%%%%%%
 \begin{definition}[Dini derivatives]\label{def:dini}
 		Define the Dini derivatives as the functions
 	\begin{flalign*}
 	\bar{\mathcal{D}}^{\pm} f (x) & = \limsup\limits_{\epsilon \rightarrow 0} \frac{\deltapm{f}{x}}{\epsilon}  \\
 	\ubar{\mathcal{D}}^{\pm} f (x) & = \liminf\limits_{\epsilon \rightarrow 0} \frac{\deltapm{f}{x}}{\epsilon}
 	\end{flalign*}
 	For the function $f $.
 \end{definition}
 
 %%%%%%%%%%%%%
 %  definition
 %%%%%%%%%%%%%
 \begin{definition}[Baire categories]
 	Let $X$ be a metric space. A set $E \subseteq X $ is of first category 
 	if it can be written as a countable union of nowhere dense sets, 
 	and is of second category if $E$ is not of first category.
 \end{definition}
 For example $\fclass{Q}{}$ and $\emptyset$ are  I category, while the class of continuous functions is of category 0.
 %%%%%%%%%%%%%
 %  definition
 %%%%%%%%%%%%%
 \begin{definition}[Baire function classes]\label{deff:bairef}
 	The function $f : \fclass{R}{} \mapsto \fclass{R}{}$ is called Baire-class I if there is a sequence of continuous functions converging to \textit{f} point-wise. 
 \end{definition}
 %%%%%%%%%%%%%%%
 
 %%%%%%%%%%%%%
 %  definition
 %%%%%%%%%%%%%
 \begin{definition}[$G_{\delta}$ and $F_\sigma$  sets]\label{def:cat2}
 	Let $X$ be a metric space. 
 	\begin{itemize}
 		\item The set  $E \subseteq X $ is $G_\delta$ if it is countable intersection of open sets, and it is $F_\sigma$ if it is countable union of closed sets.
 		\item The set  $E \subseteq X $ is meagre if it can be expressed as the union of countably many nowhere dense subsets of $X$. 
 		\item 	Dually, a co-meagre set is one whose complement is meagre, or equivalently, the intersection of countably many sets with dense interiors.	 
 	\end{itemize}
 \end{definition}
 %%%%%%%%%%%%%%%%
 
   %%%%%%%%%%%%%
 %  Section
 %%%%%%%%%%%%%
 %\section{Proof of lemma \ref{th:osc}}
 \section{The First Oscillation Lemma}\label{sec:osclem}
 
 The lemma was stated in \cite{Prodanov2017}:
 %%%%%%%%%%
 % lemma
 %%%%%%%%%%
 \begin{lemma}[Oscillation lemma]
 	\label{th:osc}	 
 	Consider the function $f: X \mapsto Y \subseteq \fclass{R}{}$.
 	Suppose that $I_{+} =\left[x, x+ \epsilon \right] \subseteq X$,  
 	$I_{-} =\left[ x - \epsilon, x \right] \subseteq X$, respectively. 
 	
 	If  $\mathrm{osc^{+}} [f] (x)=0 $ then  $f$  is {right-continuous}  at x. Conversely, if $f$  is right-continuous at x  then $  \mathrm{osc^{+}} [f] (x)=0 $.	     	
 	If  $\mathrm{osc^{-}} [f] (x)=0 $ then  $f$  is {left-continuous}   at x. Conversely, if $f$  is left-continuous at x  then $\mathrm{osc^{-}} [f] (x)=0 $.  	 
 	That is, 
 	\[
 	\llim{\epsilon}{0}{	\oscpm{f}{x}} =0 \Longleftrightarrow \llim{ \epsilon}{0}{ f(x \pm \epsilon)}= f(x)
 	\]
 \end{lemma}
 %%%%%%%%%%%%%%%%%%%%%%%%%%%%%%
 %%%%%%%%%%%%%%%%%%%%%%%%%%%%%%%%%%%%%
 Then the negation of the statement is also true.
 \begin{corollary}
 	The following two statements are equivalent
 	\[
 	\llim{\epsilon}{0}{	\oscpm{f}{x}} >0 \Longleftrightarrow \llim{ \epsilon}{0}{ f(x \pm \epsilon) \neq f(x)}
 	\]
 \end{corollary}
 %%%%%%%%%%%%%%%%%%%%
 
 \begin{proof}
 	\begin{description}
 		\item[Forward case]  
 		Suppose that  $\osc{f}{x}{}{+} =0 $. 
 		Then there exists a pair $  \mu::\delta, \ \delta \leq \epsilon $, such that
 		$
 		\osc{f}{x}{\delta}{+} \leq \mu
 		$.
 		Therefore, $f$ is bounded in $I_{+}$.
 		Since $\mu$ is arbitrary we select
 		$x^\prime $, such that 
 		\[
 		| f(x^\prime) - f(x) | =\mu^\prime \leq \mu
 		\]
 		and set $|x-  x^\prime|= \delta^\prime$. Since $\mu$ can be made arbitrary small so does $\mu^\prime$.
 		Therefore,  $f$ is (right)-continuous at $x$. 
 		
 		%%%%%%%%%%%%%%%%%%%%%%%%%%%%%%%%%%%%%
 		
 		\item[Reverse case]
 		If \textit{f} is (right-) continuous on \textit{x} then there exist a pair % related (Cauchy sequences) 
 		$\mu :: \delta$ such that		 
 		\begin{flalign*}
 		\left|  f(x^{\prime}) - f(x)\right| & < \mu/2,  \ \  \left| x^{\prime}  -  x \right|  < \delta/2 \\
 		\left|  f(x) - f(x^{\prime\prime}) \right|  & < \mu/2,  \ \ \left|   x -  x^{\prime\prime}    \right| < \delta/2
 		\end{flalign*}
 		Then we add the inequalities and by the triangle inequality we have
 		\begin{flalign*}
 		\left|  f(x^{\prime}) - f(x^{\prime\prime})\right| & \leq \left|  f(x^{\prime}) - f(x)\right|+ \left|  f(x) - f(x^{\prime\prime}) \right|  < \mu \\ 
 		\left| x^{\prime}  -  x^{\prime\prime}  \right| & \leq \left| x^{\prime}  -  x \right| + \left|   x -  x^{\prime\prime}    \right| < \delta
 		\end{flalign*}
 		However, since $ x^{\prime}$ and $x^{\prime\prime}$ are arbitrary we can set the former to correspond to the minimum and the latter to the maximum of \textit{f} in the interval.
 		Therefore, by the least-upper-bond property we can identify 
 		$ f(x^{\prime}) \mapsto \inf_\epsilon f (x) $, $f(x^{\prime\prime}) \mapsto \sup_\epsilon f (x) $.		
 		Therefore, 
 		$ \mathrm{osc}_{\delta}^{+} [f] (x) < \mu $  for   $ \left|  x^{\prime} - x^{\prime\prime}  \right| < \delta $ (for the pair $\mu :: \delta$ ).
 		Therefore, the limit is $\mathrm{osc^{+}} [f] (x)=0 $.
 	\end{description}
 	The left case follows by applying the right case, just proved, to the mirrored image of the function: $ f(-x)$.
 \end{proof}
 %%%%%%%%%%%%%%%%%%%%%%%
 
 %% 	\conflictsofinterest{The author declares no conflicts of interests.}
	 %%%%%%%%%%%
	 % Bib
	 %%%%%%%%%%%%%%
	 %\bibliographystyle{model1-num-names}  % Style BST file 
	 \bibliographystyle{plain} 
	 \bibliography{fractint}

\begin{thebibliography}{10}

\bibitem{Adda2001}
F.~Ben Adda and J.~Cresson.
\newblock About non-differentiable functions.
\newblock {\em J. Math. Anal. Appl.}, 263:721 -- 737, 2001.

\bibitem{Adda2013}
F.~Ben Adda and J.~Cresson.
\newblock Corrigendum to "{A}bout non-differentiable functions" [{J. Math.
  Anal. Appl.} 263 (2001) 721 -- 737].
\newblock {\em J. Math. Anal. Appl.}, 408(1):409 -- 413, 2013.

\bibitem{Bartle2001}
R.~Bartle.
\newblock {\em Modern Theory of Integration}, volume~32 of {\em Graduate
  Studies in Mathematics}.
\newblock American Mathematical Society, Providence, R.I., 2001.

\bibitem{Berezhnoi2003}
E.~I. Berezhnoi.
\newblock {The subspace of C[0,1] consisting of functions having finite
  one-sided derivatives nowhere}.
\newblock {\em Mathematical Notes}, 73(3--4):321--327, 2003.

\bibitem{Brown1999}
J.~Brown, U.~Darji, and E.~Larsen.
\newblock Nowhere monotone functions and functions of non-monotonic type.
\newblock {\em Proceedings of the American Mathematical Society},
  127(1):173--182, 1999.

\bibitem{Cherbit1991}
G.~Cherbit.
\newblock {\em Fractals, Non-integral dimensions and applications}, chapter
  Local dimension, momentum and trajectories, pages 231-- 238.
\newblock John Wiley \& Sons, Paris, 1991.

\bibitem{BoisReymond1875}
P.~du~Bois-Reymond.
\newblock {Versuch einer Classification der willk\"urlichen Functionen reeller
  Argumente nach ihren Aenderungen in den kleinsten Intervallen}.
\newblock {\em J Reine Ang Math}, 79:21--37, 1875.

\bibitem{Faber1909}
G.~Faber.
\newblock {\"{U}ber stetige Funktionen}.
\newblock {\em Math. Ann.}, 66:81 -- 94, 1909.

\bibitem{Feynman1948}
R.~P. Feynman.
\newblock Space-time approach to non-relativistic quantum mechanics.
\newblock {\em Reviews of Modern Physics}, 20(2):367--387, apr 1948.

\bibitem{Flood2011}
R.~Flood.
\newblock {\em Mathematics in Victorian Britain}.
\newblock OUP Oxford, 2011.

\bibitem{Fonf1999}
V.~Fonf, V.I. Gurariy, and M.I. Kadets.
\newblock {An infinite dimensional subspace of C[0,1] consisting of nowhere
  differentiable functions}.
\newblock {\em C. R. Acad. Bulg. Sci}, 52(11--12):13 -- 16, January 1999.

\bibitem{Girgensohn2001}
R.~Girgensohn.
\newblock An infinite-dimensional subspace of {C[0,1]} consisting of functions
  with no finite one-sided derivatives.
\newblock {\em Mathematica Pannonica}, 12:129--132, 2001.

\bibitem{Gleyzal1941}
A.~Gleyzal.
\newblock Interval functions.
\newblock {\em Duke Math. J.}, 8:223 -- 230, 1941.

\bibitem{Kantor2015}
Ida Kantor.
\newblock {\em Mathematics}.
\newblock American Mathematical Society, 2015.

\bibitem{Kolwankar1997}
K.~Kolwankar and A.D. Gangal.
\newblock Fractional differentiability of nowhere differentiable functions and
  dimensions.
\newblock {\em Chaos}, 6(4):505 -- 513, 1996.

\bibitem{Kolwankar2001}
K.~Kolwankar and J.~Vehel.
\newblock Measuring functions smoothness with local fractional derivatives.
\newblock {\em Frac. Calc. Appl. Anal.}, 4(3):285 -- 301, 2001.

\bibitem{Mandelbrot1982}
B.~Mandelbrot.
\newblock {\em Fractal Geometry of Nature}.
\newblock Henry Holt \& Co, 1982.

\bibitem{Milanov1977}
S.~Milanov, A.~Petrova-Deneva, A.~Angelov, and N.~Shopolov.
\newblock {\em Higher mathematics , part II}.
\newblock Technika, 1977.

\bibitem{Nelson1966}
E.~Nelson.
\newblock Derivation of the {Schr\"odinger} equation from {Newtonian}
  mechanics.
\newblock {\em Phys. Rev.}, 150:1079--1085, Oct 1966.

\bibitem{Nottale1989}
L.~Nottale.
\newblock {Fractals in the Quantum Theory of Spacetime}.
\newblock {\em Int J Modern Physics A}, 4:5047--5117, 1989.

\bibitem{Nottale2010}
L.~Nottale.
\newblock {Scale Relativity and Fractal Space-Time: Theory and Applications}.
\newblock {\em Foundations of Science}, 15:101--152, 2010.

\bibitem{Prodanov2017}
D.~Prodanov.
\newblock Conditions for continuity of fractional velocity and existence of
  fractional taylor expansions.
\newblock {\em Chaos, Solitons {\&} Fractals}, 102:236--244, sep 2017.

\bibitem{Prodanov2018}
D.~Prodanov.
\newblock Fractional velocity as a tool for the study of non-linear problems.
\newblock {\em Fractal and Fractional}, 2(1):4, jan 2018.

\bibitem{Prodanov2019}
D.~Prodanov.
\newblock Characterization of the local growth of two cantor-type functions.
\newblock {\em Fractal and Fractional}, 3(3):45, aug 2019.

\bibitem{Silva2007}
C~E Silva.
\newblock {\em {Invitation to ergodic theory}}.
\newblock Student mathematical library. American Mathematical Society,
  Providence, RI, 2007.

\bibitem{Trench2013}
W.F. Trench.
\newblock {\em {Integral Calculus of Functions of One Variable}}, chapter~3,
  pages 171 -- 177.
\newblock Trinity University, 2013.

\bibitem{Uhlenbeck1930}
G.~E. Uhlenbeck and L.~S. Ornstein.
\newblock On the theory of the {Brownian} motion.
\newblock {\em Physical Review}, 36(5):823--841, sep 1930.

\bibitem{Willard2004}
S~Willard.
\newblock {\em General Topology}.
\newblock Dover Publications Inc., 2004.

\bibitem{Yang2015}
X.-J. Yang, D.~Baleanu, and H.~M. Srivastava.
\newblock {\em Local Fractional Integral Transforms and Their Applications}.
\newblock Academic Press, 2015.

\end{thebibliography}
	\end{document}